\documentclass[10pt]{amsart}

\usepackage{amsmath,amsfonts,amsthm,amssymb}
\usepackage{color}
\usepackage[]{graphics}
\usepackage{verbatim}
\usepackage{eepic,epic}
\usepackage{epsfig,subfigure,epstopdf}
\graphicspath{{./pics/}}

\newtheorem{theorem}{Theorem}[section]
\newtheorem{lemma}{Lemma}[section]

\numberwithin{equation}{section}
\theoremstyle{definition}

\begin{document}

\newcommand{\al}{\alpha}
\newcommand{\fy}{\varphi}
\newcommand{\la}{\lambda}
\newcommand{\ep}{\epsilon}
\newcommand{\si}{\sigma}
\newcommand{\Si}{\Sigma}
\newcommand{\as}{\quad\text{as}\ }
\newcommand{\with}{\quad\text{with}\ }
\renewcommand{\for}{\quad\text{for}\ }
\newcommand{\where}{\quad\text{where}\ }
\renewcommand{\for}{\quad\text{for}\ }
\renewcommand{\Im}{\operatorname{Im}}
\newcommand{\wt}{\widetilde}

\def\tribar{\vert\thickspace\!\!\vert\thickspace\!\!\vert}
\def\PD{P(\partial_t)}
\def\PDal{P(\partial_t)}
\def\ul{u_\la}
\def\dH#1{\dot H^{#1}(\Omega)}
\def\Dal{\partial_t^\al}
\def\Om{\Omega}
\def\zK{P^\K}
\def\T{{\mathcal{T}}}
\def\I{{\mathcal{I}}}
\def\K{K}
\def\E{\mathcal{E}}
\def\H{\mathcal{H}}
\def\om{\mathcal{\omega}}
\def\calM {\mathcal{M}}
\def\M {\mathcal{M}}
\def\calS {\mathcal{S}}
\def\S {\mathcal{S}}
\def\calH {\mathcal{H}}
\def\H {\mathcal{H}}
\def\calI {\mathcal{I}}
\def\LL {\mathcal L}
\def\C {\mathcal C}
\def\ga {\gamma }
\def\de {\delta}

\title[Nonnegativity Preservation in FEM for FDE]
{On Nonnegativity Preservation in Finite Element Methods for Subdiffusion Equations}
\author[B. Jin, R. Lazarov, V. Thom\'ee, Z. Zhou]{Bangti Jin \and Raytcho Lazarov \and Vidar Thom\'ee \and Zhi Zhou}

\begin{abstract}
We consider
three types of subdiffusion
models, namely single-term, multi-term  and distributed order
fractional diffusion equations,
for which the maximum-principle holds
and which, in particular, preserve nonnegativity.
Hence the solution is nonnegative for
nonnegative initial data.
Following earlier work on the heat equation, our purpose is to study whether
this property is inherited by
certain spatially semidiscrete and fully discrete piecewise
linear finite element methods,
including the standard Galerkin method,
the lumped mass method and the finite volume element method.
It is shown that,
as for the heat equation,
when the mass matrix is nondiagonal,
nonnegativity is not preserved for small time or time-step, but may
reappear after a positivity threshold. For the lumped mass method
nonnegativity is preserved if and only if the triangulation in the finite
element space is of Delaunay type. Numerical experiments illustrate and
complement the theoretical results.
\end{abstract}

\maketitle

\section{Introduction}\label{sec:intro}

In this work we consider the numerical analysis of mathematical models arising
in subdiffusion. One simple example is the single-term subdiffusion,
for which the governing equation is given by
\begin{equation}\label{eqn:de}
  \partial_t^\alpha u-\Delta u= 0,\quad \text{in  } \Omega,\ \ t > 0,
  \with 0<\al<1,
\end{equation}
under the initial-boundary conditions
\begin{equation}\label{eqn:ibc}
u=0,\quad\text{on}\  \partial\Omega,\ \ t > 0,\with
    u(0)=v, \quad\text{in  }\Omega,
\end{equation}
where $\Omega$ is a bounded  domain in $\mathbb R^2$ which for simplicity we assume to be polygonal,
and where $v$ is a given function on $\Omega$. Here $\partial_t^\alpha u$ denotes the Caputo fractional
derivative of order $\alpha\in(0,1)$ with respect to $t$,
defined by the convolution (see \cite{KilbasSrivastavaTrujillo:2006})
\begin{equation}\label{eqn:Caputo}
  \Dal u (t) = \frac{1}{\Gamma(1-\alpha)}\int_0^t (t-s)^{-\alpha}  \frac{d}{ds}u(s)ds,
\end{equation}
where $\Gamma(x)=\int_0^\infty s^{x-1}e^{-s}ds$ is the Gamma function.
By continuity we may include the standard heat equation for $\alpha=1$.
The fractional derivative is often used to describe anomalous diffusion,
in which the mean square variance grows sublinearly with
time $t$, at a rate slower than that in a Gaussian process.
It occurs in many applied disciplines, e.g.,
electron transport in Xerox copiers, molecule transport
in membranes, and thermal diffusion on fractal domains.

By the maximum principle,
the solution $u(t)$ of the heat equation satisfies the  nonnegativity preservation
property
\begin{equation}\label{eqn:pos-contin}
v=u(0)\geq 0 \quad \mbox{ in }\Omega \quad \mbox{implies }\quad u(t)\geq 0
\ \ \mbox{ in } \Omega, \for t\geq0.
\end{equation}
It is natural to ask to what extent this property extends to
numerical approximations, and for piecewise linear
finite element methods over regular triangulations this question has
attracted some attention.
Fujii \cite{Fujii:1973} showed that for the lumped mass
(LM) method, with backward Euler time-stepping,
nonnegativity is preserved when the angles in
the triangulation are acute. Thom\'{e}e and Wahlbin
\cite{ThomeeWahlbin:2008}
showed that for the semidiscrete standard Galerkin (SG) method,
nonnegativity is not preserved for small $t>0$, but for the
LM method, it is preserved if and only if the triangulation is of Delaunay type.
Later, Schatz, Thom\'ee and Wahlbin
\cite{SchatzThomeeWahlbin:2010} analyzed
several fully discrete methods and showed that, in particular,
these properties carry over to the backward Euler method.
Recently, Thom\'{e}e \cite{Thomee:2015}
and Chatzipantelidis, Horv\'ath and Thom\'ee \cite{Chatz:2014}
extended these results to cover also the
finite volume element (FVE) method
and further
discussed the existence of a  threshold
$t_0>0$, such that nonnegativity is preserved for  $t\geq t_0$.

Our purpose in this paper is to discuss the
nonnegativity preservation  property for a class of
piecewise linear finite element approximations, including
the SG, LM and FVE methods described briefly in Section \ref{sec:semidis},
 for the more general initial--boundary value problem
\begin{align}\label{eqn:fde}
	\PDal u-\Delta u&= 0,\quad\text{in}\ \Omega, \ \ \for t>0,
\\
u&=0,\quad\text{on}\  \partial\Omega,  \for t > 0, \quad\text{and}\quad
u(0)=v,\quad\text{in  }\Omega,
\notag
\end{align}
where $\PDal u $ denotes a  fractional order differential operator of the form
\begin{equation}\label{eqn:PDal}
  \PDal u (t) = \int_0^1 \partial_t^\alpha u(t) d\nu(\alpha),
\end{equation}
with $\nu(\alpha)$  a positive measure on $[0,1]$.
These more general models, including
in addition to \eqref{eqn:de}
 multi-term and distributed order models, are reviewed briefly
in Section \ref{sec:prelim}. There
we also introduce the solution operator $E(t)$ of \eqref{eqn:fde}
and discuss its analytic properties, including asymptotic behavior.
In a series of interesting works, Luchko
\cite{Luchko:2009fcaa,Luchko2009,Luchko:2011jmaa}
established that the solution $u(t)$ of
problem \eqref{eqn:fde} also satisfies the property
 \eqref{eqn:pos-contin}.

Numerical methods for the model \eqref{eqn:fde}
and their analysis have received considerable interest, cf., e.g.,
\cite{JinLazarovLiuZhou:2015,JinLazarovSheenZhou:2015,JinLazarovZhou:2013,JinLazarovZhou:2014,LinXu:2007,MustaphaAbdallahFurati:2014}.
The weak formulation of  \eqref{eqn:fde} is  to
find $u(t) \in H^1_0(\Omega)$
for $t>0$ such that
\begin{equation*}
(\PDal u(t), \fy) + a(u(t), \fy) =0,
\quad \forall \fy \in H^1_0(\Omega),\ \ t\ge0, \with u(0)=v,
\end{equation*}
where
$a(u,\chi)=(\nabla u, \nabla \chi)$ and $(\cdot, \cdot)$
is the inner product in  $L^2(\Omega)$.
The spatially semidiscrete finite element approximation
of  \eqref{eqn:fde} is based on
a family of shape regular quasi-uniform
triangulations ${\{\T_h\}}_{0<h<1}$
of  $\Omega$ into triangles, with $h$ the maximal
length of the sides of the triangulation $\T_h$,
and the associated  finite element
space $X_h$ of continuous
piecewise linear functions over $\T_h $,
\begin{equation*}
  X_h =\left\{\chi\in H^1_0(\Om): \
	  \chi ~~\mbox{linear in}  ~~\K,
 \,\,\,\,\forall \K \in \T_h\right\}.
\end{equation*}
The spatially semidiscrete approximation
of  \eqref{eqn:fde} is then to find $u_h (t)\in X_h$ for $t\ge0$ such that
\begin{equation}\label{eqn:sgfem-i}
 {[ \PDal u_{h}(t),\chi]}+ a(u_h(t),\chi)= {0},
\quad \forall \chi\in X_h,\ t >0, \with u_h(0)=v_h,
\end{equation}
where $v_h \in X_h$ is an approximation of the
initial data $v$, and $[\cdot,\cdot]$ is a suitable
inner product on $X_h$.
With $\{P_j\}_{j=1}^N$  the
interior nodes of  $\T_h$,
and $\{\phi_j\}_{j=1}^N$  the corresponding nodal basis, we may write
\begin{equation}\label{eqn:nodal}
    u_h(t)=\sum_{j=1}^N u_j(t)\phi_j(x)
    \quad \text{and} \quad v_h=\sum_j^N v_j\phi_j(x).
\end{equation}
Setting $U(t):=(u_1(t),\ldots,u_N(t))^T$ and
$V=(v_1, \ldots, v_N)^T$, the semidiscrete problem \eqref{eqn:sgfem-i} reads
\begin{equation}\label{eqn:matrix}
	P(\partial_t)
 U(t) +  \H U(t)=0,  \for t>0,
\with U(0)=V,\where \H=\M^{-1}\S.
\end{equation}
Here $\calM =(m_{ij})$, with $m_{ij}=[\phi_j,\phi_i]$, and
$\calS=(s_{ij})$, with $s_{ij}=a(\phi_j,\phi_i)$, denote the mass and
stiffness matrices, respectively, and are both symmetric
and positive definite. The case of the semidiscrete SG method
was first developed and analyzed by Jin et al. \cite{JinLazarovZhou:2013}
for the single-term problem on a convex polygonal domain, and
then extended to multi-term and distributed-order cases in
\cite{JinLazarovLiuZhou:2015,JinLazarovSheenZhou:2015}, where
optimal error estimates with respect to the regularity of
initial data $v$ were established.

Following the analysis
in the case of the heat equation, we
discuss nonnegativity preservation for the spatially
semidiscrete methods in Section  \ref{sec:semidis}.
Introducing the solution matrix $\E(t)$ by
writing $U(t)=\E(t)V$,   nonnegativity
preservation may be expressed as
$\E(t)\ge0$, elementwise.
In Theorem \ref{thm:failSG} we show
that if the mass matrix $\M$
is nondiagonal, which holds for the SG and FVE methods,
then this cannot happen
for small $t>0$. However for the LM method this happens if
and only if $\T_h$ is Delaunay.
In Theorem \ref{theorem:SG}, we show that if $\calH^{-1} >0$,
then there exists $t_0>0$
such that $U(t) \ge0$ for all {$t\ge t_0$}, if $V\ge0$.

In Section \ref{sec:fully} we discuss positivity preservation for  fully
discrete time stepping methods with a time step $\tau$, for \eqref{eqn:matrix},
based on convolution quadrature generated
by the backward Euler method, or
\begin{equation*}
   Q_n(U) + \calH U^n=Q_n(1)V, \for  n \ge 1,\with U^0=V,
   \where U^n\approx u_h(t_n),\ t_n=n\tau.
\end{equation*}
Here $Q_n(U)=\sum_{j=0}^n\om_{n-j}U^j$
denotes the convolution quadrature approximation to
$P(\partial_t)U(t_n)$; see Section \ref{sec:fully}.
This gives the fully discrete solution in the form
\begin{equation*}
    U^n=\E_{n,\tau}\,V := (\omega_0\,\calI+  \calH)^{-1}
    \Big(\sum_{j=0}^{n-1}\omega_{j}\ V
    -\sum_{j=1}^{n-1} \omega_{n-j} U^{j}\Big), \for n\ge1,\with U^0=V.
\end{equation*}
This method  was developed and analyzed in
\cite{JinLazarovSheenZhou:2015,JinLazarovZhou:2014} for the SG approximation,
and the convergence of the fully discrete solution to the
corresponding spatially semidiscrete one, which is needed below,
is shown in the general case in Appendix A.
We show time discrete  analogues of the results in Section
\ref{sec:semidis}, and also,
in Theorem \ref{thm:threshold}, we
give an upper  bound of the positivity threshold
for special triangulations.

In the final Section 5 we present numerical examples for various domains
$\Om$, with
Delaunay and non-Delaunay triangulations $\T_h$,
to illustrate and complement
our theoretical results.

As a result of the maximum principle, in addition to \eqref{eqn:pos-contin},
it also follows that the solution operator $E(t)$ is a contraction in
the maximum-norm, i.e., $\|E(t)\|_\infty\le 1$. The question about analogues of
this for the numerical methods was also discussed in
\cite{SchatzThomeeWahlbin:2010,ThomeeWahlbin:2008}
for the heat equation,
and in Appendix B we show that the analogue of this holds in the present
case for semidiscrete and fully discrete LM approximations, provided that the
stiffness matrix $\S$ is diagonally dominant, which is not
equivalent to $\T_h$ being Delaunay.

\section{Preliminaries}\label{sec:prelim}
In this section we present preliminary material concerning
the subdiffusion model \eqref{eqn:fde} and its analytical properties.
In particular we introduce the
{solution operator $E(t)$ for $t\ge0$ and the associated
scalar kernel function $u_\la(t)$,
 and study}
its asymptotic behavior
for small and large $t$, which will be needed later.

{As indicated in Section \ref{sec:intro},}
in addition to the single-term model \eqref{eqn:de},
several more complex models for subdiffusion have been proposed,
with \eqref{eqn:Caputo} replaced by a weighted linear combination
$\PDal u$, defined in \eqref{eqn:PDal}.
We will specifically study two cases, the discrete case (denoted by
case I below), with
point masses $\{\al_j\}_{j=1}^m\subset(0,1),\ m\ge1$, where
$\alpha_m<\alpha_{m-1} <\ldots<\alpha_1$,
and the distributed
case (denoted by case II below), in which $\nu$ is completely continuous,
i.e., $d\nu(\alpha)= \mu(\alpha)d\alpha $, with $\mu$ smooth and nonnegative.
For the characteristic function of $\PDal$,
the complex function $P(z)=\int_0^1z^\alpha d\nu(\alpha)$, we then have
\begin{equation*}
  P(z) = \begin{cases}
	  \displaystyle{  \sum_{i=1}^m}
	  b_iz^{\alpha_i}&\quad \mbox{in case I},
	  \\
	  \displaystyle{\int_0^1}z^\alpha\mu(\alpha)d\alpha&
	  \quad \mbox{in case II}.
  \end{cases}.
\end{equation*}
In the discrete case, the constants $b_i$ are positive,
with $b_1=1$, which ensures the
well-posedness of the  problem \eqref{eqn:fde},  cf. \cite{LiLiuYamamoto:2015}.
It reduces to
the single-term subdiffusion model \eqref{eqn:de} when $m=1$.
In the distributed case,
the nonnegative weight function $\mu(\alpha)$ is assumed to be
H\"{o}lder continuous, with $\mu(0)\mu(1)>0$.
These conditions are sufficient for the
positivity of the solution and are different from the
ones used in the existing literature{, cf.}
\cite{Kochubei:2008,
LiLuchkoYamamoto:2014},
The general model was proposed in order to extend the
flexibility of the single-term model, in that the
underlying stochastic process may contain a host
of different Hurst exponents, cf. \cite{ChechkinGorenfloSokolov:2002}.

For the  convenience of the reader we prove
the positivity preservation property \eqref{eqn:pos-contin}
for problem \eqref{eqn:fde}, using
arguments in \cite{Luchko:2009fcaa,Luchko2009,Luchko:2011jmaa}.
We first show  an extremal principle
for the fractional differential operator $P(\partial_t)$.
\begin{lemma}\label{lem:Ext}
{Assume  $f\in C[0,T]\cap C^1(0,T]$,
and let $f$ attain its minimum  at  $t_0\in(0,T]$. Then
$\PDal f(t_0)\le0$.}
\end{lemma}
\begin{proof}
{Setting $g(s)=f(s)-f(t_0)$, we have
by integration by parts, since $0\leq g(s)=O(t_0-s)$ as $s\to t_0$,
\begin{equation*}
  \int_0^{t_0}(t_0-s)^{-\alpha}g'(s)ds =
  -t_0^{-\alpha}g(0) -
  \alpha\int_0^{t_0}(t_0-s)^{-\alpha-1}g(s)ds\leq 0,
  \quad\text{for any}\ \al\in(0,1).
\end{equation*}
Since $f'(s)=g'(s)$ we thus have
$\Dal f(t_0) \le0$,
and hence
$\PDal f(t_0)=\int_0^1 \Dal f(t_0)d\nu(\alpha)\leq 0$.}
\end{proof}
\begin{theorem}
The solution operator $E(t)$
for  \eqref{eqn:fde}
is nonnegative, i.e., $E(t)v\geq0$ for $t\geq0$, if $v\geq0$.
\end{theorem}
\begin{proof}
Assume that the minimum of $u(t)=E(t)v$
is negative and achieved at $(x_0,t_0)\in \Omega\times(0,T]$.
With $\epsilon=-u(x_0,t_0)>0$ we set
{$w(x,t)=u(x,t)-\tfrac12\epsilon(1-{t}/{T})\geq u(x,t)-\tfrac12\epsilon$.}
Then, for  $x\in\partial\Omega$ or $t=0$,
\begin{equation*}
  w(x_0,t_0) \le u(x_0, t_0)=-\ep\le -\ep + u(x,t)
  \le -\ep + w(x,t) +\tfrac12\ep = w(x,t) -\tfrac12\ep.
\end{equation*}
Hence $w(x,t)$ cannot attain its minimal value at $t=0$ of for
$x\in\partial\Omega$.
Now let the minimum of $w$ be taken at $(x_1,t_1)\in \Omega\times(0,T]$.
Then by Lemma \ref{lem:Ext},
$\PDal w(x_1,t_1)\leq 0$, and clearly $-\Delta w(x_1,t_1)\leq 0$.
Noting the identity $\Dal t = t^{1-\alpha}/\Gamma(2-\alpha)$, we  find
\begin{equation*}
  (\PDal -\Delta) u(x_1,t_1) = (\PDal -\Delta)w(x_1,t_1) -
  {\frac{\epsilon}{2T}\int_0^1 \frac{t_1^{1-\al} }{\Gamma(2-\al)}
  d\nu(\alpha)<0,}
\end{equation*}
contradicting the fact that $u(t)$ is a solution of \eqref{eqn:fde}.
\end{proof}

\medskip

Next we derive {an expression for} the solution
operator $E(t)$ for problem \eqref{eqn:fde} by
means of Laplace transformation, denoting the
Laplace transform of $u $ by $\widehat{u}(z)=(\mathcal{L}u)(z)$.
Recall that \cite[Lemma 2.24]
{KilbasSrivastavaTrujillo:2006},
\begin{equation*}
  \mathcal{L}({\Dal u})(z) = z^\alpha \widehat{u}(z) - z^{\alpha-1}u(0),
\end{equation*}
and consequently
\begin{equation*}
  \begin{aligned}
    \mathcal{L}{\PDal u}(z) & 
    =\int_0^1 {\mathcal{L}{\Dal u}(z)}d\nu(\alpha)  = \int_0^1 z^\alpha d\nu(\alpha)\widehat{u}(z) - \int_0^1 z^{\alpha-1} d\nu(\alpha) u(0)\\
    & =P(z)\widehat{u}(z) - z^{-1}P(z)u(0).
  \end{aligned}
\end{equation*}
After Laplace transformation of \eqref{eqn:fde}, with $A=-\Delta$, the
negative Laplacian with a zero Dirichlet boundary condition, and the initial
condition $u(0)=v$, we therefore obtain
\begin{equation*}
  P(z)\widehat u (z) + A\widehat u (z) = z^{-1}P(z)v,
\end{equation*}
or
\begin{equation*}
  \widehat u (z) = J(z) v,\quad \mbox{where } J(z) = z^{-1}P(z)(P(z)I+A)^{-1},
\end{equation*}
where $I$ is the identity operator.
Hence we find for the solution operator, defined by $u(t)=E(t)v$,
\begin{equation}
\label{2.0}
  E(t)v = (\mathcal{L}^{-1}J)(t)v =
  \frac{1}{2\pi \mathrm{i}} \int_{\Gamma_\sigma} e^{zt}J(z)dz v,
\end{equation}
where $\Gamma_\sigma =\{z=\sigma+\mathrm{i}\eta: \sigma>0,
\ \eta\in\mathbb{R}\}$.
Letting $\{\lambda_j\}_{j=1}^\infty$ be the eigenvalues in increasing order
and $\{\fy_j\}_{j=1}^\infty$  the corresponding $L^2(\Omega)$ orthonormal
 eigenfunctions of the operator $A$, we may also write
\begin{equation}\label{eqn:eigen-rep}
  E(t)v = \sum_{j=1}^\infty u_{\lambda_j}(t)(v,\fy_j)\fy_j,
\end{equation}
where
\begin{equation} \label{22}
u_\la(t)=	
(\LL^{-1}J_\la)(t)=\frac{1}{2\pi\mathrm{i}}
\int_{\Gamma_\si}e^{zt}J_\la(z)\,dz,\quad
J_\la(z)
=z^{-1}P(z)(P(z)
+\la)^{-1}.
\end{equation}
The scalar kernel
function $u_\lambda=u_\lambda(t)$ solves the  fractional order
initial-value problem
\begin{equation}\label{scalder}
\PDal u_\la + \la u_\la = 0,\with \ul(0)=1.
\end{equation}
Since $P(z)$ is analytic in the complex plane, cut along the negative
real axis, we obtain, after deforming the contour $\Gamma_\sigma$ in \eqref{22},
and noting that, for small $\de>0$,
\begin{equation*}
	\int_{|z|=\de}|J_\la(z)|\,|dz|  \le
	C{\delta^{-1}}\int_{|z|=\de}P(\de)\,|dz|
	{=2\pi\, C\,P(\de)}\to 0,\as \de\to0,
\end{equation*}
and hence we find that
\begin{equation}\label{23}
u_\la(t)=	
\frac1\pi\int_0^\infty e^{-st} K_\la(s)\,ds,\quad
K_\la(s)=\Im J_\la(-s)=
\frac{\la\Im P(-s)}{s|P(-s)+\la|^2}>0,
\end{equation}
where we have used
\[
\Im P(-s)=\Im\int_0^1 e^{\mathrm{i} \al\pi} s^\al\,d\nu(\al)
=\int_0^1\sin \al\pi\,s^\al\, d\nu(\al)>0.
\]
{Since
$K_\la(s)$ in \eqref{23}
is positive,
$u_\la(t)$
is completely monotone, and, in particular, monotone for $t\ge0$.}


\bigskip

We now study the asymptotic behavior
of the function $u_\lambda(t)$ defined in \eqref{22}.
Our main tool will be a special case of the
Karamata-Feller Tauberian theorem
(see \cite[Theorems 2--4, pp. 445--466]{Feller:2008}),
which we state as the following lemma.
Recall that a positive function $L(t)$ defined on $(0,\infty)$
varies slowly at infinity if for every fixed $x$,
$L(tx)/L(t)\to 1$ as $t\to\infty$.
\begin{lemma}
\label{thm:karamata-feller}
Let $u(t)$ be a monotone function with its Laplace transform
$\omega(z)=\LL u(z)$ defined for $z=s>0$,
$L(t)$ be slowly varying at infinity, and $\rho>0$. Then
\begin{equation*}
  \omega(s) \sim \frac{1}{s^\rho}L\left(\frac{1}{s}\right)
  \mbox{ as } s\to 0\quad \mbox{ implies }
  \quad u(t)\sim \frac{1}{\Gamma(\rho)}t^{\rho-1}L(t)
  \ \mbox{ as } t\to\infty.
\end{equation*}
The statement is also valid when the roles of the origin and infinity are interchanged.
\end{lemma}

We can now  determine the asymptotic behavior for the kernel function $\ul(t)$.
\begin{lemma}\label{lem:cpl_mt}
Let $\mu\in\C^\ga([0,1])$ and  $\mu(0)\mu(1)>0,\
\lambda>0$. Then
\begin{equation}\label{eqn:asymp-t=0}
	u_\la(t) =1-\la\beta_0(t)(1+o(1))
\as t\to0,
  \where \beta_0(t)=
  \begin{cases}
	  t^{\al_1}/\Gamma(1+\al_1)
	  &\quad\text{\emph{in case I}},
  \\
  t/(\mu(1)\log(t^{-1}))
  &\quad \text{\emph{in case II}},
\end{cases}
\end{equation}
and
\begin{equation}\label{eqn:asymp-t=inf}
	u_\la(t) = \la^{-1}\beta_\infty(t)(1+o(1))
  \as t\to\infty,
  \ \text{where}\ \beta_\infty(t)=
  \begin{cases}
   b_m
     t^{-\al_m}/
     \Gamma(1-\al_m),
     &\ \mbox{ \emph{in case I}},
     \\
  \mu(0) / \log t,
     &\ \mbox{ \emph{in case II}}.
  \end{cases}
\end{equation}
\end{lemma}
\begin{proof}
Since  $u_\lambda(t)$ is  monotone, Lemma
\ref{thm:karamata-feller} applies.
First, we show the assertion as $t\to0$, which is determined by the behavior of
{$\om(s)=J_\lambda(s)$ for large $s$.}
In case I (recall $b_1=1$), we have
\begin{equation*}
\begin{split}
    J_\lambda(s) &= s^{-1}
	\left(1+\la\Big(\sum_{i=1}^m b_i s^{\al_i}\Big )^{-1}\right)^{-1}
     = s^{-1} - \la s^{-\al_1-1} + o(s^{-\al_1-1}).
\end{split}
\end{equation*}
By Lemma \ref{thm:karamata-feller} this shows
the first part of assertion \eqref{eqn:asymp-t=0}.
In case II, we have
\begin{equation*}
    P(s) = \int_0^1 e^{\alpha \log s}
	\mu(1) d\alpha + \int_0^1
	e^{\alpha \log s}(\mu(\alpha)-\mu(1))d\alpha
	= \frac{\mu(1)(s-1)}{\log s}
     + \int_0^1 (\mu(\alpha)-\mu(1))e^{\alpha \log s} d\alpha.
\end{equation*}
Since $\mu\in \C^{\ga}([0,1])$ we obtain, after integration by parts,
\begin{equation*}
  \begin{aligned}
    &\big|\int_0^1 (\mu(\alpha)-\mu(1))e^{\alpha \log s}
	  d\alpha\big| \leq C\int_0^1 (1-\alpha)^\gamma
	  e^{\alpha \log s}d\alpha \leq
     C\frac{s}{( \log s)^{1+\gamma}}.
  \end{aligned}
\end{equation*}
Hence, $P(s) = \mu(1)s(\log s)^{-1} + O(s(\log s)^{-1-\gamma})$
and $P(s)^{-1}=\log s/(s\mu(1))(1+o(1))$ as $s\to\infty.$ Thus
\begin{equation*}
	J_\lambda(s) = \frac{1}{s}\frac{1}{1+\lambda/P(s)}
  = \frac{1}{s}\sum_{k=0}^\infty(-1)^k\lambda^{k}P(s)^{-k} = s^{-1} -
  \frac{\la\log s}{s^2\mu(1)} + o(s^{-2}\log s)
\quad\text{as } s\to\infty,
\end{equation*}
and now the second part of assertion  \eqref{eqn:asymp-t=0}
follows from Lemma \ref{thm:karamata-feller}.

Next we show the asymptotic behavior as $t\to\infty$.
In case I we have for $s\to 0$
\begin{equation*}
  J_\lambda(s) =
  \sum_{i=1}^m b_is^{\al_i-1}
  L(1/s)
  \quad \text{with}\quad L(t)=
  1/(\sum_{i=1}^mb_it^{-\al_i} +\la )\sim \la^{-1}\as t\to\infty.
\end{equation*}
The function $L(t)$ is slowly varying at infinity.
Hence, by Lemma \ref{thm:karamata-feller},
\begin{equation*}
  \ul(t) \sim L(t) \frac{b_m t^{-\al_m}}{\Gamma(1-\al_m)}
  \as t\to\infty,
\end{equation*}
which shows the first part of assertion \eqref{eqn:asymp-t=inf}.

We finally consider the large
time asymptotic behavior  in case II.
It follows from the splitting and using the
H\"{o}lder continuity of  $\mu(\al)$ that
\begin{equation*}
    P(s) = \int_0^1 s^\alpha \mu(\alpha) d\alpha =
    \mu(0)\int_0^1 s^\alpha d\alpha +
    \int_0^1 s^\alpha (\mu(\alpha)-\mu(0))d\alpha
        =  -\frac{\mu(0)}{\log s} + O \Big ((\log s)^{-1-\gamma} \Big )
\end{equation*}
(see also \cite[Proposition 2.2]{Kochubei:2008} for a related estimate).
Consequently,
\begin{equation*}
    J_\lambda(s)  = \lambda^{-1}s^{-1}P(s)\sum_{k=0}^\infty\lambda^{-k} P(s)^k
    = \lambda^{-1}s^{-1}\left( -\mu(0)(\log s)^{-1}(1 +o(1)\right)
	    \quad\text{as}\ s\to0.
\end{equation*}
Thus  $J_\lambda(s) = \mu(0)\la^{-1} s^{-1}  L(1/s)(1 + o(1))$
with $L(s)=1/\log s $  slowly varying at infinity, and
Lemma \ref{thm:karamata-feller} completes the proof of the lemma.
\end{proof}

\section{Nonnegativity preservation in spatially
semidiscrete methods}\label{sec:semidis}

In this section  we  describe a general class of
spatially semidiscrete methods for \eqref{eqn:fde},
and review three specific examples of such methods mentioned in
the introduction.
We then analyze  their nonnegativity preservation properties.

As described in Section \ref{sec:intro},
we consider semidiscrete approximations for problem \eqref{eqn:fde}
of the form \eqref{eqn:sgfem-i},
where $[\cdot,\cdot]$ is an inner product in
 $X_h$, approximating the usual $L^2$ inner
 product $(\cdot,\cdot)$.
We denote the stiffness matrix by $\calS = (s_{ij})$,
with $s_{ij}=a(\phi_j,\phi_i)$, and the mass matrix by
$\calM = (m_{ij})$, with $m_{ij}=[\phi_j,\phi_i]$.
The semidiscrete problem \eqref{eqn:sgfem-i} may then be written in
 matrix form as
\begin{equation*}
  \calM \PDal U + \calS U =0,\quad \forall\, t>0, \with U(0)=V,
\end{equation*}
or, equivalently, after multiplication by
$\calM^{-1}$,
\begin{equation}\label{eqn:fem-matrix}
  \PDal U + \calH U = 0 ,\quad \forall\, t>0, \with  U(0)=V,
\where
\H=\M^{-1}\S.
\end{equation}
In  analogy with \eqref{2.0},
we find, for the solution matrix of \eqref{eqn:fem-matrix},
\begin{equation*}
  \E(t) = (\mathcal{L}^{-1}\mathcal{J})(t)=
  \frac{1}{2\pi\mathrm{i}}\int_{\Gamma_\sigma}e^{zt}\mathcal{J}(z)dz,
  \quad \mbox{with } \mathcal{J}(z) = z^{-1}P(z)(P(z)\mathcal{I}+\calH)^{-1},
\end{equation*}
where $\mathcal{I}$ is the identity matrix.
Since both the stiffness matrix $\calS$
and the mass matrix $\calM$ are symmetric
positive definite, the eigenvalue problem
$\calS \fy=\lambda\calM\fy$ has a complete system of eigenvectors $\fy_i$
with positive eigenvalues $\lambda_i$, so that
$\calH=\calM^{-1}\calS=\mathcal{Q}^{-1}
\mathcal{D}\mathcal{Q}$,
where the rows of $\mathcal{Q}$ are the eigenvectors $\fy_i$
and the diagonal elements of the diagonal
matrix $\mathcal{D}$ are the eigenvalues $\lambda_i$.
Thus, corresponding to the eigenfunction expansion
\eqref{eqn:eigen-rep}, we have
\begin{equation}
	\label{3.E}
  \E(t)=\mathcal{Q}^{-1}\mathcal{D}_{\Lambda(t)}\mathcal{Q},
   \with \mathcal{D}_{\Lambda(t)}=\mathrm{diag}(u_{\lambda_i}(t)).
\end{equation}

We now briefly review our three examples of finite element methods
and the corresponding inner
products $[\cdot,\cdot]$ in $X_h$.
The first example is
the standard Galerkin (SG) method,
with the standard $L^2(\Omega)$ inner product,
 i.e., we choose $[\cdot,\cdot]=(\cdot,\cdot)$.

Our second example is the lumped mass (LM) method,
using
\begin{equation}\label{LM-matrix}
[w,\chi]
=(w, \chi)_h = \sum_{\K \in \T_h}  Q_{\K,h}(w \chi), \with
Q_{\K,h}(f) = \tfrac13\,|\K|\sum_{j=1}^{3} f(\zK_j) \approx \int_\K f dx,
\end{equation}
where $\zK_j$, $j=1, 2,3$, are the vertices of the
triangle $K\in\T_h$ and $|K|$ is its area.
In this case the mass matrix ${\calM}$
is diagonal with positive diagonal elements.

Our third example is
the finite volume element (FVE) method, cf.
\cite{ChatzLazarovThomee:2013,ChouLi:2000}, which is  based
on a discrete version of
 the local conservation law
\begin{equation}
	\label{FVE}
  \int_V\PDal u(t) dx - \int_{\partial V}\frac{\partial u}{\partial n}ds
  = 0\quad \mbox{for }t\geq 0,
\end{equation}
valid for any $V\subset\Omega$ with a piecewise smooth boundary $\partial V$,
with $n$ the unit outward normal to $\partial V$.  The discrete method
then requires \eqref{FVE} to be satisfied
for $V=V_j,\ j=1,\dots,N$, which  are disjoint so called
{\it control volumes} associated with the
nodes $P_j$ of $\T_h$.
It can be recast as a Galerkin method, by
letting
\begin{equation*}
Y_h= \left\{ \fy\in L^2(\Omega): \fy|_{V_j} = \text{constant},
~~j=1,2,...,N; \ \fy = 0~~\text{outside}~~ \cup_{j=1}^N V_j  \right\},
\end{equation*}
 introducing the interpolation operator
$J_h:\C(\Omega)\to Y_h$
 by
$(J_hv)(P_j)=v(P_j)$, $j=1,\ldots,N$, and
then defining the inner product
$\langle\chi,\psi\rangle=(\chi,J_h\psi)$ for all $\chi,\psi
\in X_h$. The FVE method then corresponds to \eqref{eqn:sgfem-i} with
$[\cdot,\cdot]=\langle\cdot,\cdot\rangle$.
\medskip

We recall that an edge $e$ of the triangulation $\mathcal{T}_h$ is
called a Delaunay edge if the sum of the angles
$\psi_1$ and $\psi_2$ opposite $e$ is  $\le\pi$,
and  that $\mathcal{T}_h$ is a \textit{Delaunay triangulation} if all
interior edges are Delaunay. A node of $\T_h$ is said to
be \textit{strictly interior} if all its neighbors are
interior nodes, and $\T_h$ is \textit{normal} if it has a
strictly interior node, $P_j$ say, such that any neighbor of
$P_j$ has a neighbor which is not a neighbor of $P_j$.
A symmetric, positive definite matrix with non-positive off
diagonal entries is called a \textit{Stieltjes} matrix.
\medskip

We now  turn to our main goal, to determine whether the
nonnegativity preservation property
\eqref{eqn:pos-contin} remains valid for
the semidiscrete problem, i.e.,  if $\E(t)\ge0$ for $t\ge0$.
We shall first discuss the general case, and then the LM method.
Our first result states that the semidiscrete method
\eqref{eqn:fem-matrix}
does not satisfy \eqref{eqn:pos-contin}
 in general,
if the mass matrix $\calM$ is \textit{nondiagonal}, in the
sense that $m_{ij}>0$ for all neighbors $P_i,P_j$.

\begin{theorem}\label{thm:failSG}
Assume that the triangulation $\T_h$ is normal and
that the mass matrix $\calM$ is nondiagonal.
Then the solution matrix $\E(t)$ cannot be nonnegative for all $t>0$.
\end{theorem}
\begin{proof}
Assume that $\E(t)\ge0$ for $t\ge0$.
By \eqref{3.E} and
Lemma \ref{lem:cpl_mt}, we have
	\begin{equation}
		\label{3b0}
		\E(t)= \I-\beta_0(t)\H(1+o(1))
		\as t\to0,
\end{equation}
and, since $\beta_0(t)>0$ for small $t$,
the nonnegativity of $\E(t)$ implies $h_{ij}\le 0$ for all $i\neq j$.
Let $P_j$ be a strictly interior node as in the definition of $\T_h$ being normal.
We shall show that $h_{ij}=0$ for $i\ne j$.
Consider first the case
that $P_i$ is not a neighbor of $P_j$, so that $m_{ij}=s_{ij}=0$.
Since $\calS=\calM\calH$,
\begin{equation*}
  0=s_{ij}=\sum_{k=1}^N m_{ik}h_{kj}=\sum_{k=1,k\neq j}^N m_{ik}h_{kj}.
\end{equation*}
Since $h_{kj}\leq 0$ for $k\neq j$, we have $m_{ik}h_{kj}\leq 0$, $k\neq j$, and thus $m_{ik}h_{kj}=0$ for $k\neq j$,
and in particular, $h_{ij}=0$. When $P_i$ is a neighbor of $P_j$,
it has a neighbor $P_k$ which is not
a neighbor of $P_j$, and hence $s_{kj}=\sum_{l\neq j}m_{kl}h_{lj}=0$, which
implies $h_{ij}=0$ since $m_{ki}>0$,
in view of the assumption that $\calM$ is nondiagonal. Thus,
\begin{equation*}
  s_{ij}=\sum_{k=1}^N m_{ik}h_{kj}= m_{ij}h_{jj},\quad i=1,\ldots,N,
\end{equation*}
i.e., the $j$th columns of $\calS$ and $\calM$ are proportional, which
contradicts the facts that $\sum_{i=1}^Ns_{ij} = 0$ for a strictly
interior node $P_j$ and $\sum_{i=1}^Nm_{ij}>0$.
\end{proof}

Next we show a  nonnegativity preservation result
in the general case, for large time.
\begin{theorem}\label{theorem:SG}
Suppose $\calH^{-1}>0$.  Then
there exists a $t_0>0$ such that $\E(t)>0$
for all $t>t_0$.
\end{theorem}
\begin{proof}
As before, it follows from the asymptotic behavior
of the function $u_\lambda(t)$, cf. Lemma \ref{lem:cpl_mt}, that
	\begin{equation*}
	\E(t) = \beta_\infty(t)\H^{-1}(1+o(1))
  \as t\to\infty.
\end{equation*}
Since by  assumption $\calH^{-1}>0$,
we find that
$\E(t)$ is positive
for large $t$.
\end{proof}

Theorem \ref{thm:failSG} covers the SG and FVE methods,
and Theorem \ref{theorem:SG} applies also to the LM case.
We recall from \cite{Chatz:2014}
that $\calS^{-1}>0$ if $\S$ is
a Stieltjes matrix, and that $\S$ is a Stieltjes matrix
if and only if $\T_h$ is a Delaunay triangulation.
In particular, if $\T_h$ is Delaunay, then $\S^{-1}>0$, and hence also
$\H^{-1}>0$, but $\H^{-1}$ may be positive
also for some non-Delaunay triangulations.

In the case of the LM method we may prove the following
sharper result.
\begin{theorem}\label{thm:semi-lm}
Let $\E(t)$ be the solution matrix of the LM approximation. Then $\E(t)\ge 0$
for all $t\ge 0$ if and only if the
triangulation $\T_h$ is Delaunay.
\end{theorem}
\begin{proof}
If $\E(t)\ge 0$ for all $t\ge 0$ we conclude as
in the proof of Theorem \ref{thm:failSG} that $h_{ij}\le 0$ for
all $i\neq j$, Consequently, since $\M$ is a positive
diagonal matrix, $s_{ij}\le0$ for all $i\neq j$, so
that $\S$ is Stieltjes and hence $\T_h$  Delaunay.
The proof of the converse statement requires the corresponding result
for the fully
 discrete solution matrix $\E_{n,\tau}$,
and the convergence of the fully discrete
solution matrix to the semidiscrete one.
These two results are shown in Theorem \ref{thm:LM-fully} and Lemma \ref{4conv},
respectively, and assuming their validity we find
 $\E(t)=\lim_{n\to\infty}\E_{n,t/n}\geq0$.
\end{proof}

	We remark that in the single-term case,
the converse part of Theorem \ref{thm:semi-lm}
also follows from the representation
$\E(t)=E_\al(-t^\al\H)$
of the solution
shown in \cite{SakamotoYamamoto:2011},
where
$E_{\alpha}(z)=\sum_{k=0}^\infty z^k/\Gamma(k\alpha +1)$
is the Mittag-Leffler function.
Since, by
\cite{Pollard:1948},
$E_{\alpha}(-t)$
is completely monotone,
 Bernstein's theorem implies that
$E_{\alpha}(-y) =  \int_{0}^\infty e^{-xy} \,d\sigma(x)$, with $\si(x)$
a positive measure.
Recalling  from
\cite{Chatz:2014} that
$e^{-t\H}\ge0$
for $t\ge0$ when $\T_h$ is Delaunay,
we conclude
\[
\E(t)=E_\al(-t^\al\H)
=\int_0^\infty e^{-x\,t^\al\,\H}d\si(x)\ge0,\quad\text{for}\ t\ge0.
\]

\section{Nonnegativity preservation of a fully discrete method}
\label{sec:fully}

Now we discuss the preservation of nonnegativity for a fully discrete
method for the subdiffusion
model \eqref{eqn:fde}, developed in
\cite{JinLazarovZhou:2014,JinLazarovSheenZhou:2015}, based on
applying convolution quadrature
\cite{Lubich:1988} to 
 the semidiscrete problem \eqref{eqn:fem-matrix}.

 Using the definition of the Caputo fractional derivative, we deduce that
$ P(\partial_t) \fy(t)= P(\partial_t)\left(\fy(t)-\fy(0)\right) $,
and hence problem \eqref{eqn:fem-matrix} may then be rewritten as
\begin{equation}\label{4fdm}
  P({\partial_t}) (U(t)-V) + \calH U(t) = 0,
  \for t>0, \with U(0)=V.
\end{equation}

For the discretization in time of
\eqref{4fdm} we introduce a time step  $\tau >0$,
and set $t_n=n\tau$, $n\ge0$. For  {  a smooth function
$\fy$}, the approximation $Q_n(\fy)$, of $P(\partial_t) \fy(t_n)$,
is then given by
\begin{equation*}
	Q_n(\fy) = \sum_{j=0}^n \omega_{n-j}\fy^j, \for  n\ge0,
    \where \fy^j=\fy(t_j).
\end{equation*}
The weights $\{\omega_j\}$ are generated by the characteristic
polynomial $(1-\xi)/\tau$ of the
backward Euler (BE) method, i.e.,
$\sum_{j=0}^\infty \omega_j \xi^j= P((1-\xi)/\tau).$
The fully discrete scheme is then given by \cite{JinLazarovZhou:2014}:
\begin{equation*}
   Q_n(U) + \calH U^n=Q_n(1)V, \for  n \ge 1,\with U^0=V.
\end{equation*}
Noting that $\om_n(U^0-V)=0$,
the fully discrete solution can be explicitly represented by
\begin{equation}\label{eqn:BE_SG}
    U^n=\E_{n,\tau}\,V := (\omega_0\,\calI+  \calH)^{-1}
    \Big(\sum_{j=0}^{n-1}\omega_{j} V
    -\sum_{j=1}^{n-1} \omega_{n-j} U^{j}\Big), \for n\ge1,\with U^0=V.
\end{equation}

Clearly this relation defines a sequence of rational functions
$r_{n,\tau}(\la),\ n=0,1,2,\dots,$
such that $\E_{n,\tau}=r_{n,\tau}(\H)$, and analogously to
\eqref{3.E} we may write, with $\la_j$ the eigenvalues of $\H$,
\begin{equation}
	\label{4.E}
  \E_{n,\tau}=\mathcal{Q}^{-1}\mathcal{D}_{\Pi^n}\mathcal{Q},
   \with \mathcal{D}_{\Pi^n}=\mathrm{diag}(r_{n,\tau}(\lambda_i)).
\end{equation}

We now  turn to  nonnegativity preservation property
for the method \eqref{eqn:BE_SG}. We first show
some properties of the
quadrature weights $\{\omega_j\}$.

\begin{lemma}\label{lem:quad1prop}
Let the weights $\{\omega_j\}$ be generated by the backward Euler method.
Then
\[
		(\text{i})\qquad	\omega_0>0 \ \ \text{and}\ \
		\omega_j<0\ \for\  j\ge1,
	\quad\text{and}\qquad (\text{ii})\qquad
	\displaystyle \sum_{j=0}^n \omega_j>0, \for\  n\ge1.
	\qquad\qquad
\]
\end{lemma}

\begin{proof}
It suffices to show the assertion for the single-term model
($m=1$ and $\al_1=\al$), since the more general
case follows by linearity.
Property (i) follows directly from the explicit formula
\begin{equation*}
\omega_j=\tau^{-\alpha}(-1)^j\binom{\al}{j}
     =\tau^{-\alpha}(-1)^j\frac{\alpha(\alpha-1)\ldots(\alpha-j+1)}{j!}.
\end{equation*}
For (ii), we claim that the partial sum satisfies the relation
\begin{equation}\label{eqn:sum1}
   \sum_{j=0}^n \omega_j =\tau^{-\alpha}\frac{(n+1)(-1)^n}{\al}
  \binom{\al}{n+1},\for n\ge0.
\end{equation}
We show the claim by mathematical induction.
First, we note that $\omega_0=\tau^{-\alpha}$ satisfies
\eqref{eqn:sum1} with $n=0$. Next, assuming
\eqref{eqn:sum1}  holds up to $n-1$, we find
\begin{equation*}
   \sum_{j=0}^n \omega_j = \sum_{j=0}^{n-1} \omega_j +\omega_n
   =\tau^{-\alpha}\frac{n(-1)^{n-1}}{\al}
   \binom{\al}{n}
   +\tau^{-\alpha}(-1)^n \binom{\al}{n}
    =\tau^{-\alpha}\frac{(n+1)(-1)^n}{\al}
   \binom{\al}{n+1}.
\end{equation*}
Thus the relation \eqref{eqn:sum1} holds  for all $n \ge0$,
which completes the proof of the lemma.
\end{proof}

\begin{lemma}\label{lem:fullyequiv1}
Let
 $\E_{1,\tau}=\om_0\,(\omega_0\,\calI+ \calH)^{-1}\ge0$, where
$\om_0=P(\tau^{-1})$.
Then  $\E_{n,\tau}\ge0$  for  $n\ge1$.
\end{lemma}
\begin{proof}
	Let $U^0=V\ge0$
and let $U^n$ be defined by
\eqref{eqn:BE_SG} for $n\ge1$.
For a proof by induction,
assume that $U^j\ge0$ for $j\le n-1$.
By assumption,
 $(\omega_0\,\calI+ \calH)^{-1}\ge0$,
and by
Lemma \ref{lem:quad1prop} we have
$ \sum_{j=0}^{n-1}
\omega_{j} V -\sum_{j=1}^{n-1} \omega_{n-j} U^{j} \ge0.$
Hence $U^n=\E_{n,\tau}V\ge0$, which completes the proof.
\end{proof}

Next we state the following fully discrete analogue of Theorem \ref{thm:failSG}.
\begin{theorem}
Assume that the triangulation $\T_h$ is normal and the mass matrix $\calM$ is nondiagonal. Then the
solution matrix $\E_{1,\tau}$ cannot be nonnegative for  small $\tau>0$.
\end{theorem}
\begin{proof}
From the proof of Lemma \ref{lem:cpl_mt} we find as $  \tau\to0$
\begin{equation*}
\omega_0^{-1} =
P(\tau^{-1})^{-1}
=\wt\beta_0(\tau)(1+o(1)) 
\where \, \,
\wt\beta_0(\tau)
=\begin{cases}
\displaystyle     \tau^{\al_1}, &\ \mbox{in case I},
\\
{\displaystyle     \mu(1)^{-1}{\tau\log (\tau^{-1})},}
&\ \mbox{in case II},
\end{cases}
\end{equation*}
and hence
\begin{equation}\label{eqn:firststep}
	\E_{1,\tau}= (\calI+\om^{-1}_0\calH)^{-1}
	=\I-\wt\beta_0(\tau)\H\,(1+o(1)). 
\end{equation}
Since $\wt\beta_0(\tau)>0$,
the  argument in the proof of Theorem \ref{thm:failSG} then
completes the proof.
\end{proof}

We now show some results concerning positivity
thresholds for the fully discrete method.
\begin{theorem}\label{thm:SG-full} The following statements hold.
\begin{itemize}
  \item[(i)] If $\calH^{-1}>0$, then there exists a
	  $\tau_0>0$ such that $\E_{n,\tau}\ge0$ for
	  $\tau\ge\tau_0,\ n\ge 1$.
  \item[(ii)] If $\E_{1,\tau_0}\geq0$, then
	  $\E_{n,\tau}\geq0$ for $\tau\geq \tau_0,\ n\ge1$,
	  and
$\calH^{-1}\geq0$.
\end{itemize}
\end{theorem}
\begin{proof}
	If $\calH^{-1}>0$, then by continuity,
$(\omega_0\calI+\calH)^{-1}\ge0$ for $\omega_0$ small, i.e.,
for $\tau$ large. In particular, there exists $\tau_0\ge0$ such that
$\E_{1,\tau}\ge0$ for $\tau\ge\tau_0$, and hence $\E_{n,\tau}\ge0$
for $\tau\ge\tau_0,\ n\ge1$,
	by Lemma \ref{lem:fullyequiv1}.

For part (ii), we note that if $\E_{1,\tau_0}\ge0$,
then $(\wt\om_0\I+\H)^{-1}\ge0$, with $\wt\om_0=P(\tau_0^{-1})$.
We show that then
$(\omega_0\calI+\calH)^{-1}\geq0$ for
$\omega_0\in[0,\tilde\omega_0]$.
In fact, with $\delta=\tilde\omega_0-\omega_0>0$, we may write
\begin{equation*}
(\omega_0\,\calI +\calH)^{-1}
  = (\tilde\omega_0\calI + \calH -\delta \calI)^{-1}
  =(\tilde \omega_0\calI+\calH)^{-1}(\calI-\mathcal{K})^{-1},
\with \mathcal{K}=\delta(\tilde\omega_0\,\calI+\calH)^{-1}.
\end{equation*}
Note that $\mathcal{K}\geq 0$ by assumption,
and if $\delta$ is so small that
for some matrix norm $|\cdot|$, $|\mathcal{K}|
=\delta |(\tilde\omega_0\,\calI+\calH)^{-1}|<1$,
then $(\calI-\mathcal{K})^{-1}=\sum_{j=0}^\infty \mathcal{K}^j\geq0$,
and therefore, $(\omega_0\,\calI+\calH)^{-1}\geq0$.
But if $(\omega_0\,\calI+\calH)^{-1}\geq0$
for $\omega_0\in(\tilde{\tilde{\omega}}_0,\tilde\omega_0]$,
with $\tilde{\tilde{\omega}}_0\geq0$,
then $(\tilde{\tilde{\omega}}_0\,\calI+\calH)^{-1}\geq0$.
Hence, by repeating the argument, we deduce $(\omega_0\calI+\calH)^{-1}\geq0$
for some $\omega_0<\tilde{\tilde{\omega}}_0$,
and thus the smallest such $\tilde{\tilde{\omega}}_0$ has to be 0.
Since $P(\tau^{-1})$ is monotone,
this means that $\E_{1,\tau}\ge0$ for $\tau\ge\tau_0$, and thus
shows the first part of (ii), by
Lemma \ref{lem:fullyequiv1}.
The last part of (ii) now follows from
$\H^{-1}=
\lim_{\om_0\to0}(\om_0\I+\H)^{-1}$.
\end{proof}

We note that if $\H^{-1}>0$ and $\om_0$ is the largest number
such that $(\om_0\I+\H)^{-1}\ge0$, then in the single-term case
the positivity threshold equals $\om_0^{-1/\al}$. For $\om_0>1$, this
number increases with $\al\in(0,1]$, with
its largest value for $\al=1$, i.e., for the heat equation,
and tends to zero as $\al\to0$.

The following result gives  a somewhat more precise estimate
for the positivity threshold $\omega_0$,
in the case of ``strictly"  Delaunay triangulations.
\begin{theorem}\label{thm:threshold}
If $s_{ij}<0$ for all neighbors $P_i,P_j$
and $\calM$ is nondiagonal, then $\E_{n,\tau}\geq 0$
for all $n\ge1$ if
$\omega_0 \leq  \max_\mathcal{N} {|s_{ij}|}/{m_{ij}},$
where $\mathcal{N}=\{(i,j):\ P_i,P_j\ \mbox{neighbors}\}$.
\end{theorem}
\begin{proof}
If the assumptions hold, we have $\omega_0\,m_{ij}+s_{ij}<0$ for
$j\neq i$, so that $\omega_0\,\calM +\calS$ is a Stieltjes
matrix. Hence $(\om_0\,\I+\H)^{-1}=(\omega_0\,\calM+\calS)^{-1}\M\geq0$,
and thus $\E_{n,\tau}\geq 0$ for $n\ge1$, by Lemma \ref{lem:fullyequiv1}.
\end{proof}

In particular, for the single-term  model,  in the case of a
quasiuniform  family of triangulation $\mathcal{T}_h$, with
$\psi_1+\psi_2\leq\gamma<\pi$ for all edges $P_{i}P_j$, we have
$\E_{n,\tau}\geq0$ for $n\ge1$,
if $\tau\geq c_\alpha h^{2/\alpha}$, with $c_\al>0$.

We finally show the following fully discrete analogue of
Theorem \ref{thm:semi-lm}
for the LM method.
\begin{theorem}\label{thm:LM-fully}
Let $\E_{n,\tau}$ be the solution matrix of the LM method.
Then $\E_{n,\tau}\geq 0$ for all $\tau>0$,
$n\ge1$, if  $\T_h$ is Delaunay, and, conversely,
if $\E_{1,\tau}\ge0$ for all $\tau>0$, then
$\T_h$ is Delaunay.
\end{theorem}
\begin{proof}
If $\mathcal{T}_h$ is Delaunay, then
the stiffness matrix $\calS$ is Stieltjes, and so is
$\omega_0 \calM+ \calS$, for any $\om_0\ge0$. Hence
$(\omega_0\calI + \calH)^{-1}
= (\omega_0 \calM +\calS)^{-1} \calM\ge 0$.
and, by Lemma \ref{lem:fullyequiv1}, $\E_{n,\tau}\ge0$ for all
$\tau>0, n\ge1$.
Conversely, if $\E_{1,\tau}\geq 0$
for all $\tau>0$, then \eqref{eqn:firststep} holds  for $\tau$ small.
Hence $h_{ij}\le0$ and thus $s_{ij}\le0$ for
$P_i.P_j$ neighbors, so that $\calS$ is Stieltjes and $\T_h$ is Delaunay.
\end{proof}

\section{Numerical experiments and discussions}\label{sec:numer}
In this section we present numerical experiments to illustrate and
to complement our theoretical findings.
We consider the following five
domains with different triangulations, cf. Fig \ref{fig:meshes}:
\begin{itemize}
 \item[(a)] The unit square $\Omega=(0,1)^2$, partitioned by
	 uniform Delaunay triangulations;
 \item[(b)] The unit square $\Omega=(0,1)^2$, partitioned
	 into	  a structured family of non-Delaunay triangulations;
 \item[(c)] The L-shaped domain $\Omega={(0,1)^2\backslash([1/2,1)
 \times(0,1/2])}$,
	 partitioned by unstructured Delaunay triangulations;
 \item[(d)] The unit disk $\Omega=\{(x,y): x^2+y^2 < 1\}$,
	 partitioned by unstructured Delaunay triangulations;
 \item[(e)] The unit square $\Omega=(0,1)^2$, partitioned
	into
	 non-Delaunay triangulations, obtained from
the uniform triangulations in Example (a) with an additional subdivision
	 of one boundary triangle.

\end{itemize}

Our numerical experiments cover single-term ($m=1$,
 $\alpha_1=\alpha\in(0,1)$),
multi-term (with $m=2$,  $1>\alpha_1>\alpha_2>0$),
and distributed order cases
of problem \eqref{eqn:fde},
to be made specific below.
For each example, we investigate the preservation  of
nonnegativity for the spatially semidiscrete and
the fully discrete SG, FVE and LM methods. In the { third} case,
of course, calculations are only needed for the non-Delaunay
triangulations in (b) and (e).
For our numerical calculations of the solution
matrices $\E(t)$ we
use the eigenvalue decomposition \eqref{3.E},
where the $u_{\la_j}(t)$
are defined as inverse Laplace transforms in \eqref{22}.
For the latter we
employ a numerical algorithm from
\cite{WeidemanTrefethen:2007},
deforming the contour $\Gamma_\sigma$ to the left branch
 of a hyperbola in the complex plane with the parametric representation
$z(\xi):= \lambda(1+\sin(\mathrm{i} \xi-\psi ))$,
$\lambda>0$, $\psi\in(0,\pi/2)$ and $\xi \in \mathbb{R}$.
The resulting integral is then approximated by a truncated
trapezoidal rule, which, for an appropriate choice of the parameters
$\la$ and $\psi$,
is  exponentially convergent.

\begin{figure}[hbt!]
\subfigure[ uniform mesh]{
\includegraphics[trim = .1cm .1cm .1cm .1cm, clip=true,width=.18\textwidth]{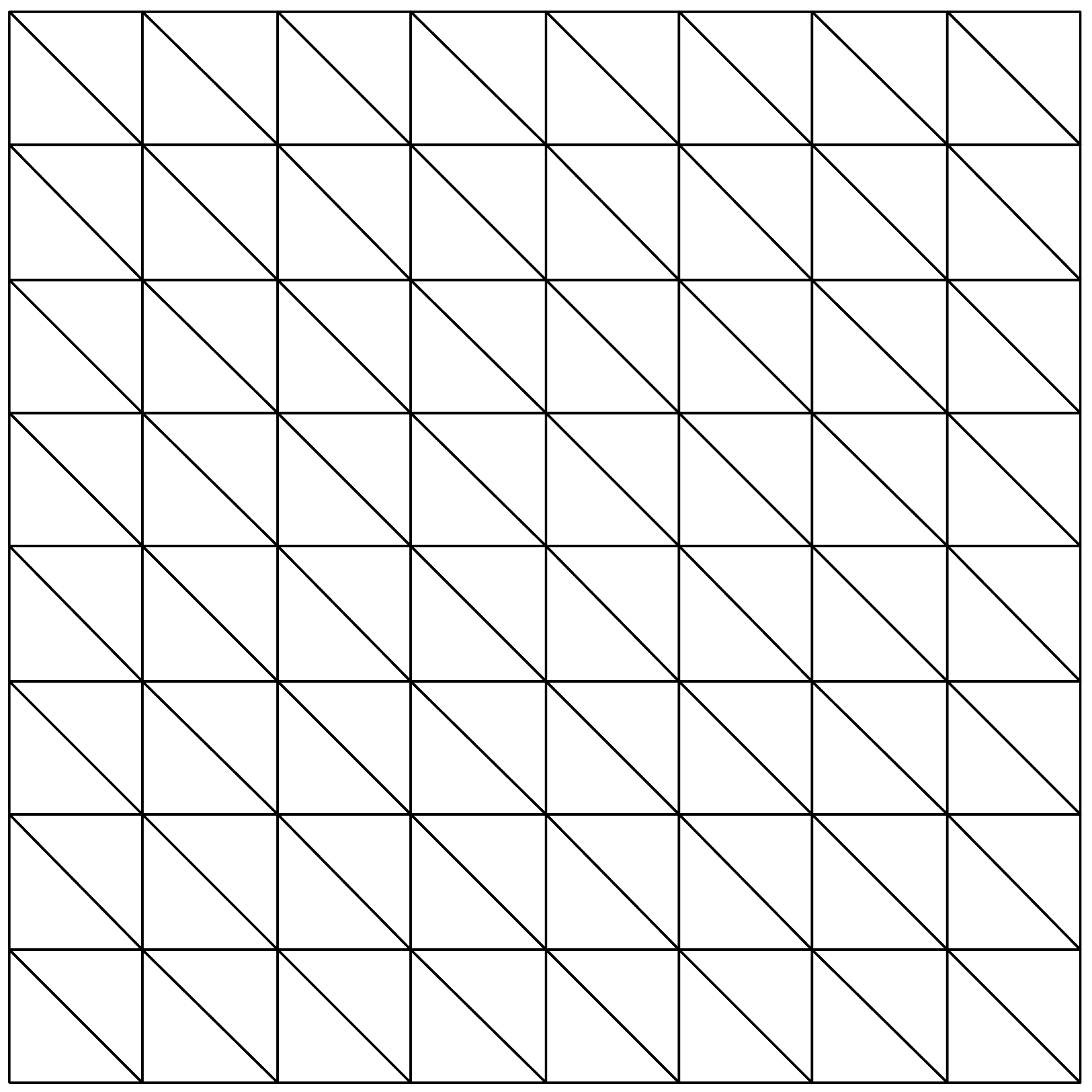}
}
\subfigure[nonDelaunay mesh]{
\includegraphics[trim = .1cm .1cm .1cm .1cm, clip=true,width=.18\textwidth]{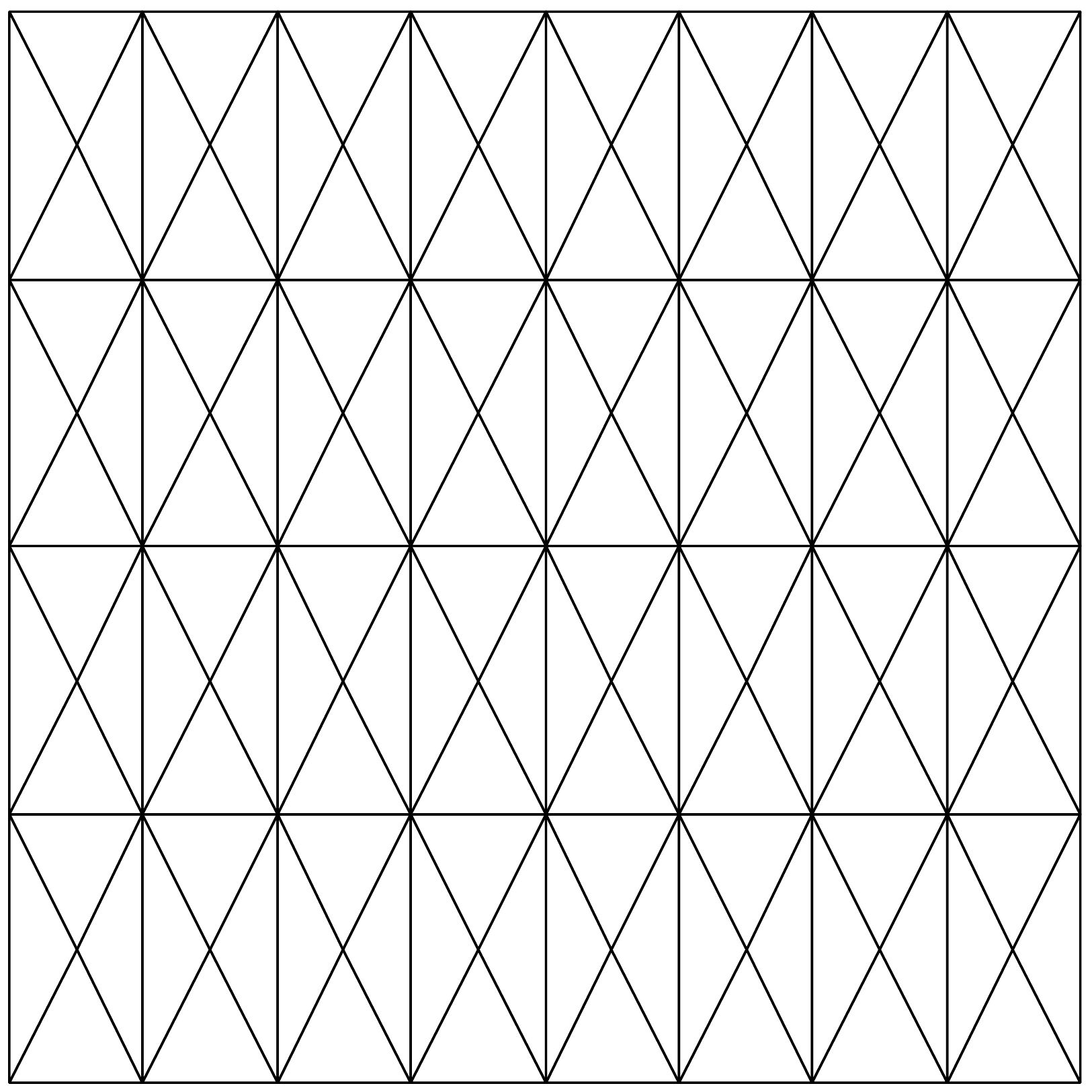}
}
\subfigure[L-shape domain]{
\includegraphics[trim = .1cm .1cm .1cm .1cm, clip=true,width=.18\textwidth]{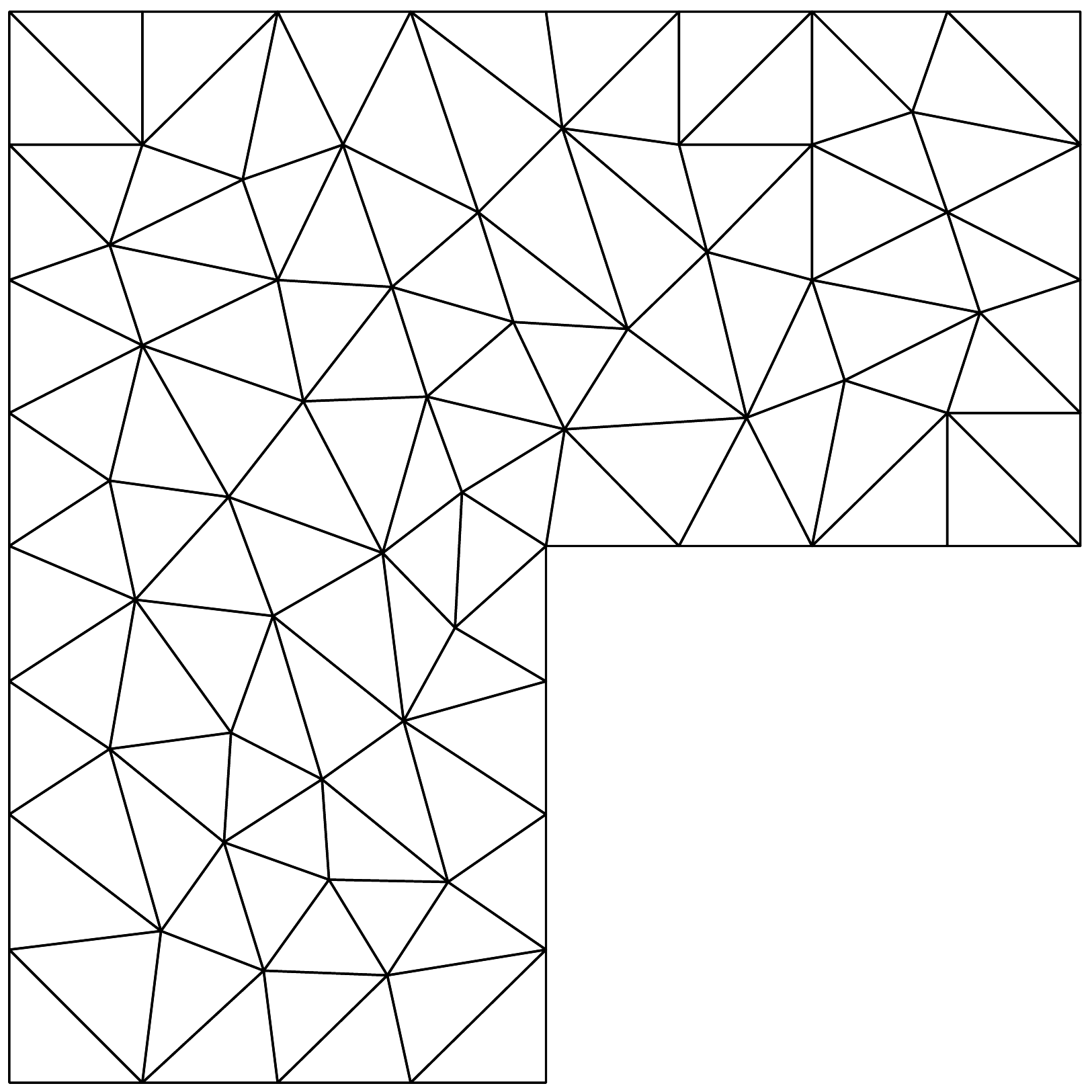}
}
\subfigure[unit disk]{
\includegraphics[trim = .1cm .1cm .1cm .1cm, clip=true,width=.18\textwidth]{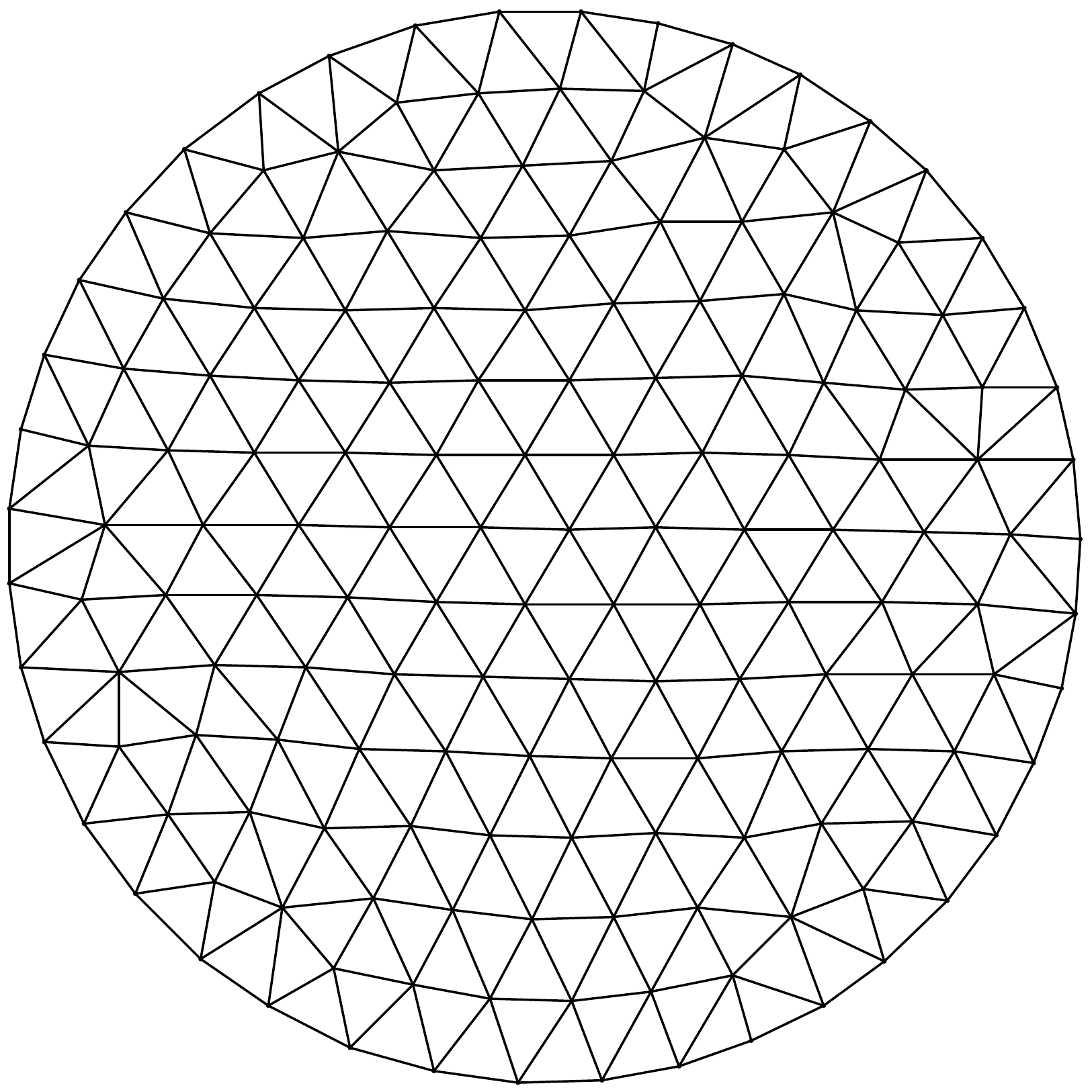}
}
\subfigure[nonDelaunay mesh]{\includegraphics[trim = .1cm .1cm .1cm .1cm, clip=true,width=.18\textwidth]{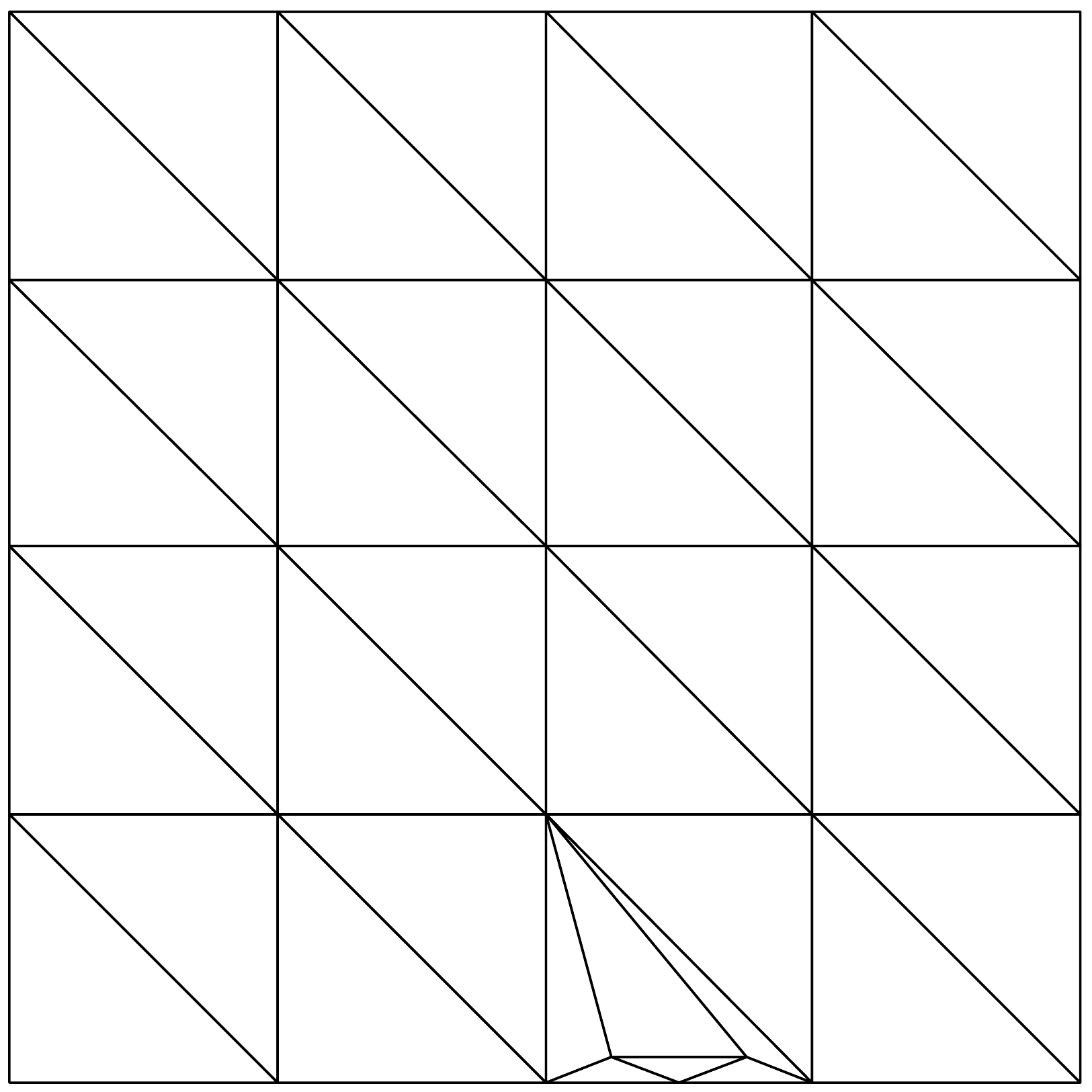}
}
 \caption{The domains with triangulations
 of Examples  (a), (b), (c), (d) and (e). }
\label{fig:meshes}
\end{figure}

{\it Example (a).} {As} in \cite{Chatz:2014}
we begin by studying a family of uniform triangulations
$\T_h $, constructed as follows. For given $M>0$ we divide
the sides of $\Om$ into $M$ equal subintervals
of length $h_0=1/M$, thus dividing $\Omega$
into $M^2$ small squares. By means of parallel diagonals
in the small squares we then  obtain a uniform triangulation,
with $h=\sqrt{2}h_0$. Note that for such
triangulations, the stiffness matrix $\calS=(s_{ij})$
has $s_{ii}=4$ and $s_{ij}=-1$ if $P_iP_j$ are vertical
or horizontal neighbors, with $s_{ij}=0$ otherwise.
In particular $\S$ is Stieltjes, with $\calS^{-1}>0$, and
hence $\calH^{-1}=\calS^{-1}\calM>0$ for the SG, LM and FVE methods.
Thus the results in Sections \ref{sec:semidis}
and \ref{sec:fully} concerning nonnegativity preservation
for large $t$ and $\tau$ apply.

In Fig. \ref{fig:semi_a}, we plot the smallest entry of the
solution matrix $\E(t)$ for the
semidiscrete SG, LM and FVE methods with {$h_0=0.100$}.
The LM method preserves nonnegativity, and
in Table \ref{tab:semi-fully_a} we present threshold
values $\hat t_0$ and $\tilde t_0$
for the semidiscrete SG and FVE methods, respectively,
for $M=10$, $20$ and $40$. Their behavior is
in agreement with  Theorems \ref{thm:failSG},  \ref{theorem:SG}  and \ref{thm:semi-lm}.
Like in \cite{Chatz:2014} the thresholds are smaller for the FVE methods
 than for the SG methods and  decrease with the fractional order $\alpha$ in the single-term case.
We also observe that the thresholds decrease with $h$. For the fully discrete method,
the positivity thresholds $\hat\tau_0$ and $\tilde\tau_0$ of $\E_{1,\tau}$ are also given
in Table \ref{tab:semi-fully_a}. Note that the triangulations do not satisfy the condition in Theorem
\ref{thm:threshold}.
Nevertheless, following the argument in \cite[Section 5.1]{Chatz:2014},
lower bounds for the positivity thresholds may be derived in the fully discrete case.
By Table \ref{tab:semi-fully_a} it appears that,
for the single-term model, this is of
order $O(h^{2/\alpha})$, and
 the multi-term model
shows a similar behavior.

\begin{figure}[th!]
\subfigure[single-term, $\al=0.5$]{
\includegraphics[trim = 1.1cm .1cm .1cm .1cm, clip=true,width=5cm]{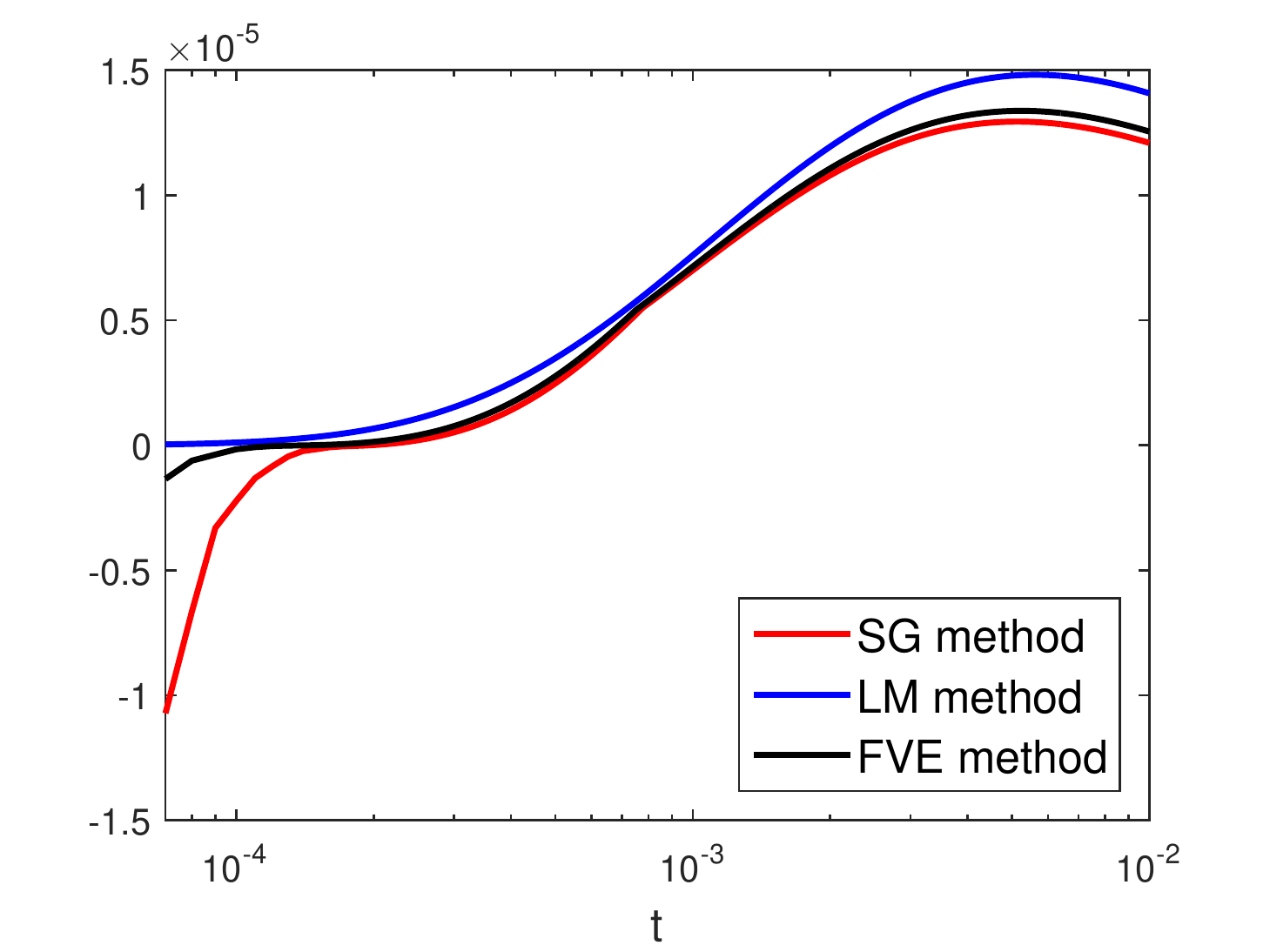}
}
\subfigure[multi-term, $\al_1=0.5$, $\alpha_2=0.2$]{
\includegraphics[trim = .9cm .1cm .1cm .1cm, clip=true,width=5cm]{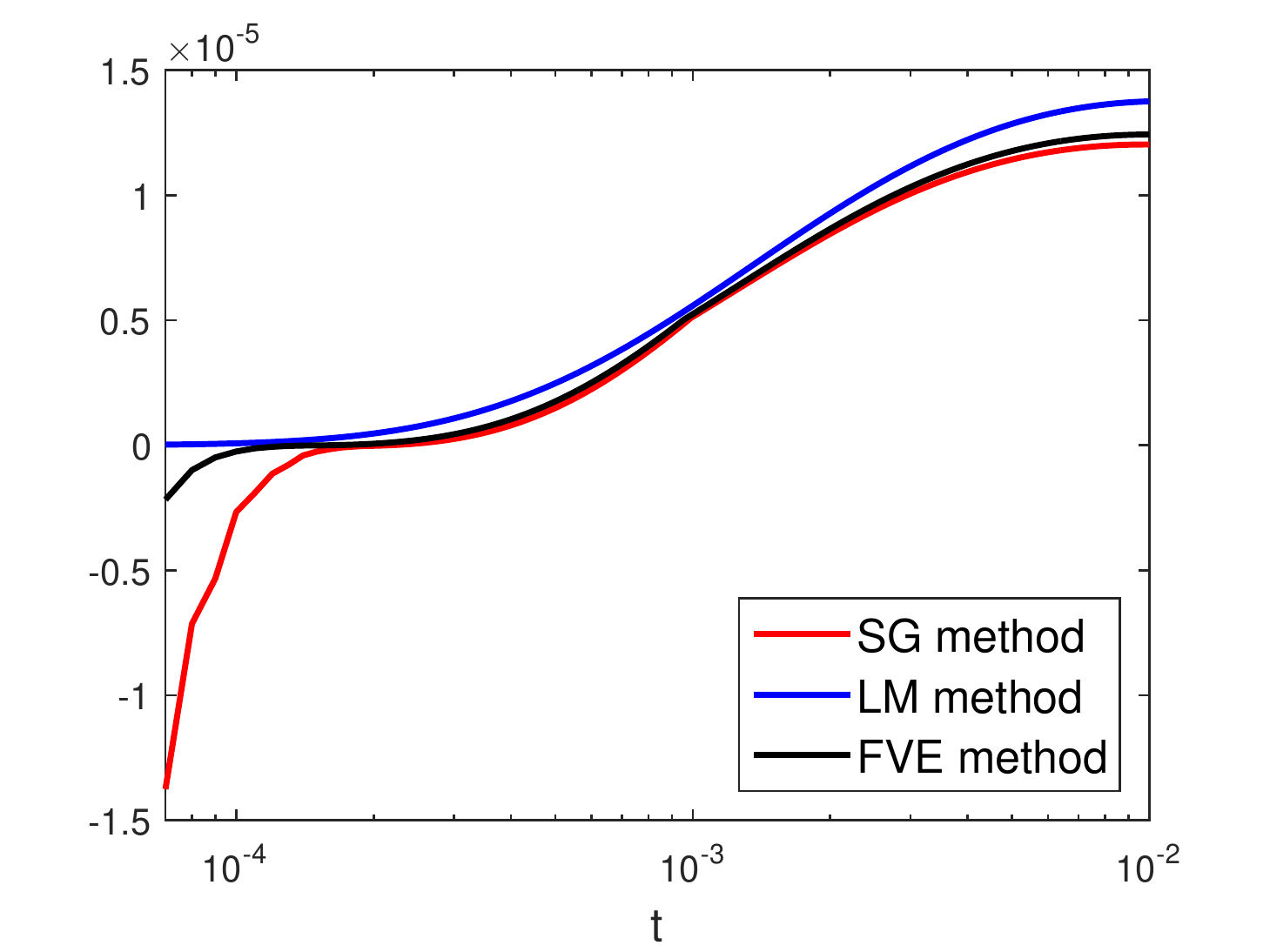}
}
\subfigure[distributed order, $\mu(\al)=e^{\alpha}$]{
\includegraphics[trim = .9cm .1cm .1cm .1cm, clip=true,width=5cm]{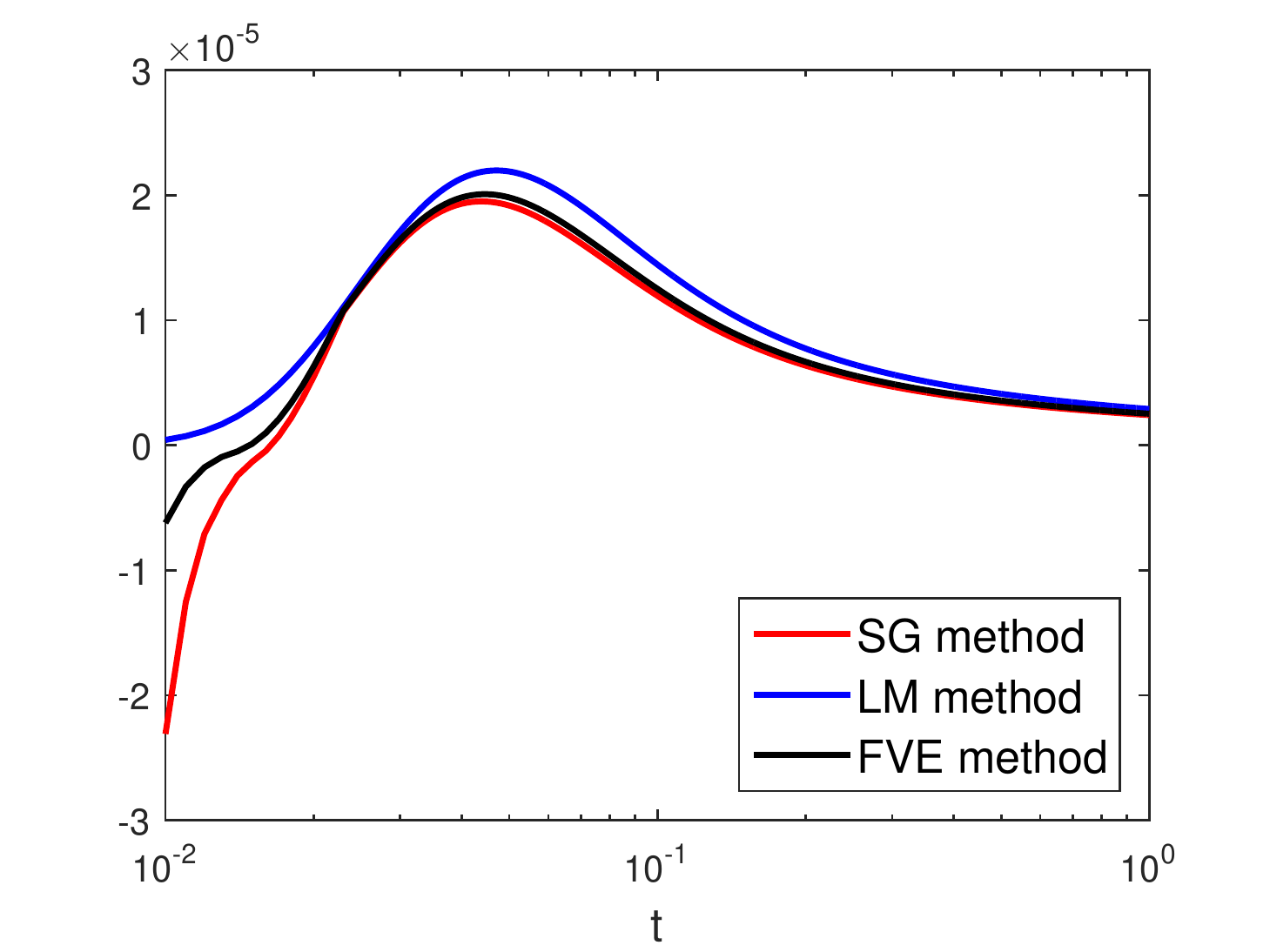}
}
\caption{The evolution of the smallest entry of
$\E(t)$ for the semidiscrete SG, LM and FVE methods
for Example (a), with $h_0=0.100$. }
\label{fig:semi_a}
\end{figure}

\begin{table}[h]
\begin{center}{
\caption{Positivity thresholds for Example (a),
for the semidiscrete and fully discrete SG and FVE  methods.}
\label{tab:semi-fully_a}
\begin{tabular}{|c|c|c|cc|cc|cc|cc|cc|cc|}
\hline
     & & &\multicolumn{4}{c|}{single-term ($\alpha$)}                          & \multicolumn{4}{c|}{multi-term ($\alpha_1,\alpha_2$)}                           & \multicolumn{2}{c|}{distributed ($\mu(\alpha)$)}                    \\ \cline{4-13}
     &$h_0$& $h$& \multicolumn{2}{c|}{$0.5$} & \multicolumn{2}{c|}{$0.75$} & \multicolumn{2}{c|}{$(0.5,0.2)$}
     & \multicolumn{2}{c|}{$(0.75,0.2)$} & \multicolumn{2}{c|}{$e^{\al}$} \\ \cline{4-13}
&   & & $\hat t_0\backslash\hat \tau_0$      & $\tilde t_0\backslash\tilde \tau_0$
&$\hat t_0\backslash\hat \tau_0$  & $\tilde t_0\backslash \tilde \tau_0$ & $\hat t_0\backslash\hat\tau_0$
& $\tilde t_0\backslash \tilde\tau_0$ & $\hat t_0\backslash\hat \tau_0$
 & $\tilde t_0\backslash \tilde\tau_0$ & $\hat t_0\backslash\hat \tau_0$           & $\tilde t_0\backslash \tilde\tau_0$  \\ \hline
&0.100 & 0.141 &1.96e-4 & 1.46e-4  &  7.42e-3  &  6.41e-3&  2.11e-4 & 1.56e-4&   7.56e-3 &   6.51e-3  & 1.64e-2 & 1.48e-2   \\
SD&0.050&0.071& 3.74e-5&  2.74e-5 &  3.54e-3 &  2.98e-3 &3.86e-5 &2.82e-5  & 3.57e-3 &3.00e-3  & 1.03e-2 &  9.12e-3  \\
&0.025 & 0.035 &1.36e-5&  1.29e-5 &  1.44e-3 &  1.19e-3  & 1.41e-5 & 1.16e-5 &1.43e-3  & 1.23e-3 & 5.42e-3& 4.64e-3 \\ \hline
&0.100 &0.141 &2.85e-5   & 1.98e-5  & 9.33e-4 & 7.32e-4 &3.11e-5  & 2.14e-5 & 9.61e-4  &7.50e-4 &  2.01e-3 & 1.63e-3  \\
FD&0.050 &0.071& 1.81e-6  & 1.26e-6  & 1.49e-4 &1.17e-4 &1.88e-6 & 1.30e-6 & 1.50e-4 &1.18e-4 & 4.17e-4  & 3.39e-4  \\
&0.025 & 0.035&1.37e-7  & 8.24e-8  & 2.65e-5&1.89e-5 &1.39e-7 & 8.36e-8 & 2.67e-5&1.90e-5 & 9.80e-5 & 7.42e-5  \\ \hline
\end{tabular}}
\end{center}
\end{table}

{\it Example (b).} We next consider a family of non-Delaunay triangulations $\T_h$
from \cite[Section 5.2]{Chatz:2014}, constructed as follows. Let $M$ be
a positive integer, $h_0=1/(2M)$, and set $x_i=ih_0$,
$i=0,\ldots,2M$, and $y_j=2jh_0$ for $j=0,\ldots,M$. This
divides the domain $\Omega$ into small rectangles
$(x_i,x_{i+1})\times (y_j,y_{j+1})$, and we now  connect
the nodes $(x_i,y_j)$ with $(x_{i+1},y_{j+1})$,
and $(x_{i+1},y_j)$ with $(x_i,y_{j+1})$.
For this triangulation, $h=2h_0$, and all vertical edges are non-Delaunay,
cf., Fig. \ref{fig:meshes}(b). The
stiffness and mass matrices are given
analytically in \cite{Chatz:2014}.

Since  $\T_h$ is non-Delaunay,
the LM solution  matrices $\E(t)$ and $\E_{1,\tau}$ cannot be nonnegative
for small
$t>0$ and $\tau>0$, by Theorems \ref{thm:semi-lm}
and \ref{thm:LM-fully}, respectively.
However, even though the stiffness matrix
$\S$ is not Stieltjes, in our computations
 $\mathcal{S}^{-1}>0$,  and hence by Theorems
 \ref{theorem:SG} and \ref{thm:SG-full},
there exist positivity thresholds for all three methods,
cf. Table \ref{tab:semi-fully-bb}.
For all three subdiffusion models, the positivity thresholds are still  smallest for the LM method  and
largest  for the SG method, and the thresholds decrease with $h$.

\begin{table}[h]
\begin{center}{
\caption{Positivity thresholds for Example (b)
for the semidiscrete and fully discrete SG, LM and FVE  methods.}
\label{tab:semi-fully-bb}
\begin{tabular}{|c|c|c|ccc|ccc|ccc|}
\hline
     & && \multicolumn{3}{c|}{single-term ($\alpha$)}                          & \multicolumn{3}{c|}{multi-term ($\alpha_1,\alpha_2$)}                           & \multicolumn{3}{c|}{distributed ($\mu(\alpha)$)}                    \\ \cline{4-12}
     &$h_0$& $h$&\multicolumn{3}{c|}{$0.5$} & \multicolumn{3}{c|}{$(0.5,0.2)$} & \multicolumn{3}{c|}{$e^{\al}$} \\ \cline{4-12}
& &  & $\hat t_0\backslash\hat \tau_0$ & $\bar t_0\backslash\bar \tau_0$     & $\tilde t_0\backslash\tilde \tau_0$
   &$\hat t_0\backslash\hat \tau_0$  & $\bar t_0\backslash\bar \tau_0$   & $\tilde t_0\backslash \tilde \tau_0$
    &$\hat t_0\backslash\hat\tau_0$  & $\bar t_0\backslash\bar \tau_0$   & $\tilde t_0\backslash \tilde\tau_0$  \\ \hline
&0.100 & {0.200}& 2.59e-4 & 1.17e-4 & 2.26e-4  & 2.99e-4&   1.30e-4 &   2.58e-4  & 1.09e-2 & 7.02e-3 & 1.03e-2   \\ 
SD&0.050  &0.100&  3.97e-5 & 1.15e-5 & 3.14e-5  & 4.13e-5&   1.20e-5 & 3.27e-5  & 7.85e-3 & 3.52e-3 & 7.13e-3    \\ 
&0.025  & 0.050 &1.77e-5 & 1.27e-5 & 1.68e-5  & 1.78e-5&   1.38e-5 & 1.74e-5  & 4.70e-3 & 3.55e-3 & 4.41e-3    \\\hline
&0.100 & 0.200 &6.20e-4 & 2.21e-4 & 5.21e-4  & 7.73e-4&   2.60e-4 &   6.42e-4  & 1.25e-2 & 6.71e-2 & 1.13e-2      \\ 
FD&0.050 & 0.100 &3.88e-5 & 1.38e-5 & 3.26e-5  & 4.26e-5&   1.48e-5 &   3.57e-5  & 2.41e-3 & 1.32e-3 & 2.17e-3    \\ 
&0.025 & 0.050 &2.42e-6 & 8.64e-7 & 2.03e-6  & 2.52e-5&   8.90e-7 &   2.12e-6  & 4.91e-4 & 2.74e-4 & 4.44e-4    \\ \hline  
\end{tabular}}
\end{center}
\end{table}

\bigskip

{\it Examples (c) and (d).}
In these examples unstructured Delaunay triangulations are generated using
the public domain mesh generators  \texttt{Triangle}
\cite{shewchuk1996triangle} and \texttt{DistMesh} \cite{Persson:2004} for
(c) and (d), respectively.
The LM method  preserves nonnegativity
for $t\ge0$, and  the positivity
thresholds for the SG and FVE methods, in
Tables \ref{tab:semi-fully_c} and \ref{tab:semi-fully_d},
exhibit a behavior similar to that for the uniform triangulations
in Example (a). 
For example, in the single-term subdiffusion,
the positivity threshold for the
fully discrete schemes shows a $O(h^{2/\alpha})$
dependence on  $h$.

\begin{table}[h]
\begin{center}
{\caption{Positivity thresholds for Example (c)
for the semidiscrete and fully discrete SG and FVE methods.}
\label{tab:semi-fully_c}
\begin{tabular}{|c|c|cc|cc|cc|cc|cc|cc|}
\hline
    &  & \multicolumn{4}{c|}{single-term ($\alpha$)} & \multicolumn{4}{c|}{multi-term ($\alpha_1,\alpha_2$)} & \multicolumn{2}{c|}{distributed ($\mu(\alpha)$)} \\ \cline{3-12}
    &$h$ & \multicolumn{2}{c|}{$0.5$} & \multicolumn{2}{c|}{$0.75$} & \multicolumn{2}{c|}{$(0.5,0.2)$} &
\multicolumn{2}{c|}{$(0.75,0.2)$} & \multicolumn{2}{c|}{$e^{\al}$} \\ \cline{3-12}
&     & $\hat t_0\backslash\hat \tau_0$           & $\tilde t_0\backslash\tilde \tau_0$         &$\hat t_0\backslash\hat \tau_0$
   & $\tilde t_0\backslash\tilde \tau_0$ & $\hat t_0\backslash\hat \tau_0$           & $\tilde t_0\backslash\tilde \tau_0$ & $\hat t_0\backslash\hat \tau_0$
       & $\tilde t_0\backslash\tilde \tau_0$ & $\hat t_0\backslash\hat \tau_0$           & $\tilde t_0\backslash\tilde \tau_0$  \\ \hline
&0.198 & 1.17e-4   & 8.52e-5  & 5.04e-3 & 4.17e-3 &1.25e-4  &8.72e-5 & 5.13e-3&4.23e-3 &1.17e-2  & 1.02e-2\\ 
SD&0.101 & 2.05e-5  & 1.47e-5  & 2.28e-3 &1.83e-3 &2.12e-5 & 1.51e-5 & 2.29e-3&1.83e-3 &  7.38e-3  & 6.12e-3\\ 
&0.051 & 1.21e-5  & 1.05e-5  & 1.26e-3 &1.09e-3 &1.23e-5 & 1.12e-5 & 1.27e-3 &1.11e-3 & 4.11e-3  & 3.63e-3\\ \hline 
&0.198 & 4.36e-5   &3.03e-5  & 1.24e-3 & 9.72e-4 &4.81e-5  & 3.31e-5 & 1.28e-3&1.00e-3 &2.57e-3 & 2.08e-3\\ 
FD&0.101 & 3.07e-6  & 2.09e-6  & 2.11e-4 &1.64e-4 &3.21e-6 & 2.18e-6 & 2.14e-4&1.65e-4 &5.61e-4  &4.52e-5\\ 
&0.051 & 2.50e-7  & 1.73e-7  & 3.97e-5 &3.10e-5 &2.55e-7 & 1.76e-7 & 3.99e-5&3.12e-5 & 1.38e-4   &1.12e-4 \\ \hline 
\end{tabular}
}\end{center}
\end{table}

\begin{table}[h]
\caption{Positivity thresholds for Example (d)
for the semidiscrete and fully discrete SG and FVE methods.}
\label{tab:semi-fully_d}
\vspace{-.2cm}
\begin{center}
{
\begin{tabular}{|c|c|cc|cc|cc|cc|cc|cc|}
\hline
    &  & \multicolumn{4}{c|}{single-term ($\alpha$)} & \multicolumn{4}{c|}{multi-term ($\alpha_1,\alpha_2$)} & \multicolumn{2}{c|}{distributed ($\mu(\alpha)$)} \\ \cline{3-12}
    &$h$ & \multicolumn{2}{c|}{$0.5$} & \multicolumn{2}{c|}{$0.75$} & \multicolumn{2}{c|}{$(0.5,0.2)$} &
\multicolumn{2}{c|}{$(0.75,0.2)$} & \multicolumn{2}{c|}{$e^{\al}$} \\ \cline{3-12}
&     & $\hat t_0\backslash\hat \tau_0$           & $\tilde t_0\backslash\tilde \tau_0$         &$\hat t_0\backslash\hat \tau_0$
   & $\tilde t_0\backslash\tilde \tau_0$ & $\hat t_0\backslash\hat \tau_0$           & $\tilde t_0\backslash\tilde \tau_0$ & $\hat t_0\backslash\hat \tau_0$
       & $\tilde t_0\backslash\tilde \tau_0$ & $\hat t_0\backslash\hat \tau_0$           & $\tilde t_0\backslash\tilde \tau_0$  \\ \hline
&0.20   & 6.68e-4  & 4.29e-4  & 1.78e-2 & 1.38e-2 &7.38e-4 & 4.67e-4 & 1.83e-2 & 1.41e-2 & 3.66e-2  & 2.99e-2\\ 
SD&0.10 & 9.11e-5  & 6.12e-5  & 7.12e-3 & 5.42e-3 &9.46e-5 & 6.42e-5 & 7.18e-3 & 5.46e-3 & 1.95e-2  & 1.58e-2\\ 
&0.05   & 5.68e-5  & 5.24e-5  & 3.58e-3 & 3.29e-3 &6.00e-5 & 5.53e-5 & 3.59e-3 & 3.34e-3 & 1.02e-2  & 9.36e-3\\ \hline 
&0.20   & 9.39e-5  & 5.95e-5  & 2.07e-3 & 1.53e-3 &1.06e-4 & 6.64e-5 & 2.16e-3 & 1.58e-3 & 4.04e-3  & 3.09e-3\\ 
FD&0.10 & 6.06e-6  & 3.88e-6  & 3.32e-4 & 2.47e-4 &6.40e-6 & 4.07e-6 & 3.38e-4 & 2.50e-4 & 8.25e-4  & 6.40e-4\\ 
&0.05   & 4.03e-7  & 2.50e-7  & 5.45e-5 & 3.96e-5 &4.12e-7 & 2.55e-7 & 5.46e-5 & 3.98e-5 & 1.80e-4  & 1.37e-4 \\ \hline 
\end{tabular}
}\end{center}
\end{table}

{\it Example (e).} We finally consider  a family of
non-Delaunay triangulations $\T_h$ of the unit square,
studied in Dr\u{a}g\u{a}nescu et al.
\cite{Druaguanescu:2005}.
Using the uniform triangulation of $\Om$ in Example (a)  with
$h_0=1/M$, we further subdivide one triangle at the boundary by introducing
three extra nodes, $Q=(1/2+h_0/4,h_0\ep)$, $R=(1/2+3h_0/4,h_0\ep)$
with $\epsilon=10^{-3}$,
 and $P=(h_0/2,0)$,
cf. Fig. \ref{fig:meshes}(e).

Since $\T_h$ is non-Delaunay, by
Theorems \ref{thm:semi-lm} and \ref{thm:LM-fully}
the LM methods do not preserve nonnegativity  for small positive
$t$ and $\tau$.
It was shown in
\cite{Druaguanescu:2005}
that $\calS^{-1} \ngeq 0$, and hence
$ \calH^{-1}\ngeq0$
for the LM methods.
In particular, by Theorem \ref{thm:SG-full},
the fully discrete method cannot
preserve positivity for large $\tau$.
It can be seen that $ \calH^{-3}>0$, i.e., that
$\H^{-1}$ is eventually positive,
which is sufficient to ensure the
existence of a positivity threshold for $\E(t)$
in the case of
the heat equation (see \cite{Chatz:2014}).
In contrast, in the fractional case,
no positivity threshold
appears to exist
for the semidiscrete LM method
for  the single-term model,
which shows a drastic difference between anomalous and
normal diffusion.
 The evolution of the smallest entry of $\E(t)$
for the single-term case with $\al=0.5$ and $e^{- \calH t}$
(corresponding to the heat equation) for the semidiscrete methods
is shown in Fig. \ref{fig:semi_e}.

However, even for this somewhat pathological triangulation,
in our computations the matrix $\calH^{-1}$ turns out to be
positive for the SG and FVE methods, and
hence, by Theorems \ref{theorem:SG} and \ref{thm:SG-full},
both the semidiscrete and
the fully discrete methods preserve nonnegativity for large $t$ and $\tau$.
Table \ref{tab:semi-fully_e} shows positivity thresholds
for the FVE and SG methods, and we note that again
they are smaller for the FVE method than for the SG method and
decrease with $h$.

\begin{figure}[th!]
\subfigure[$\E(t)=e^{-\calH t}$]{
\includegraphics[trim = 1cm .15cm .3cm 0cm, clip=true,width=4.5cm]{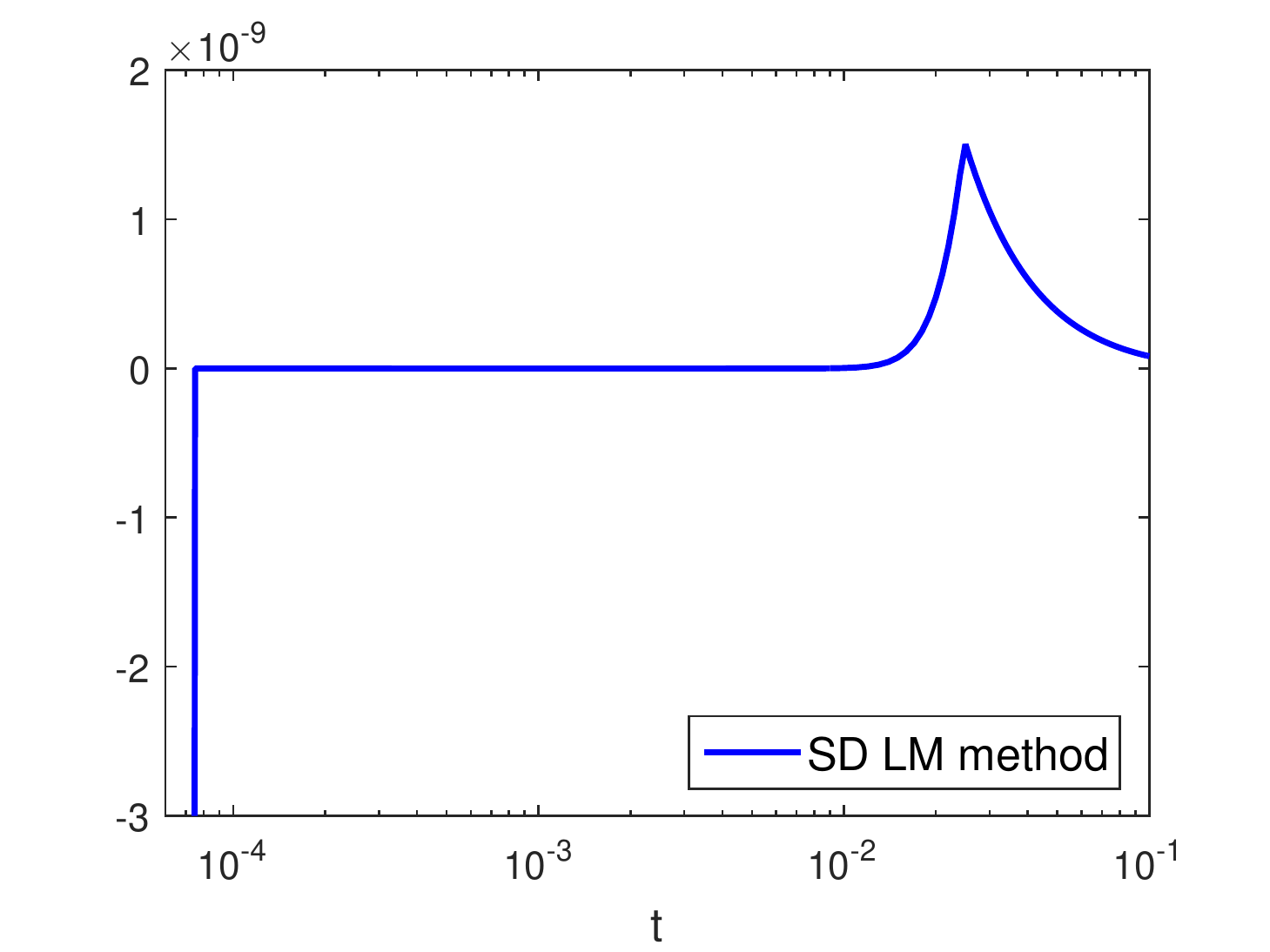}
}
\subfigure[$\E(t)=E_{\al}(- \calH t^\alpha)$, $\al=0.5$]{
\includegraphics[trim = 1.1cm .15cm .3cm 0cm, clip=true,width=4.5cm]{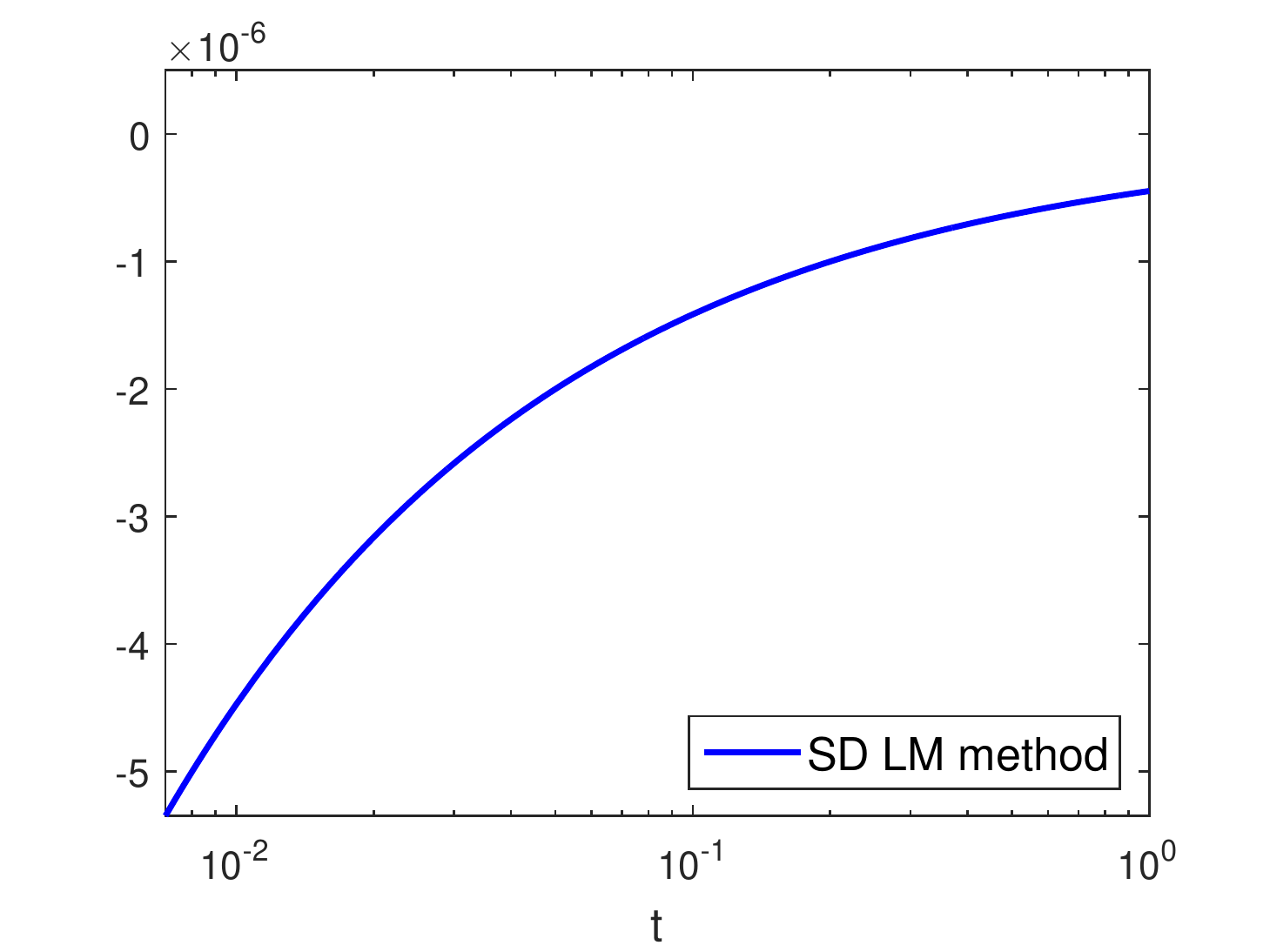}
}
\subfigure[$(\tau^\al I+ \calH)^{-1}$, $\al=0.5$]{
\includegraphics[trim = 1cm .15cm .3cm 0cm, clip=true,width=4.5cm]{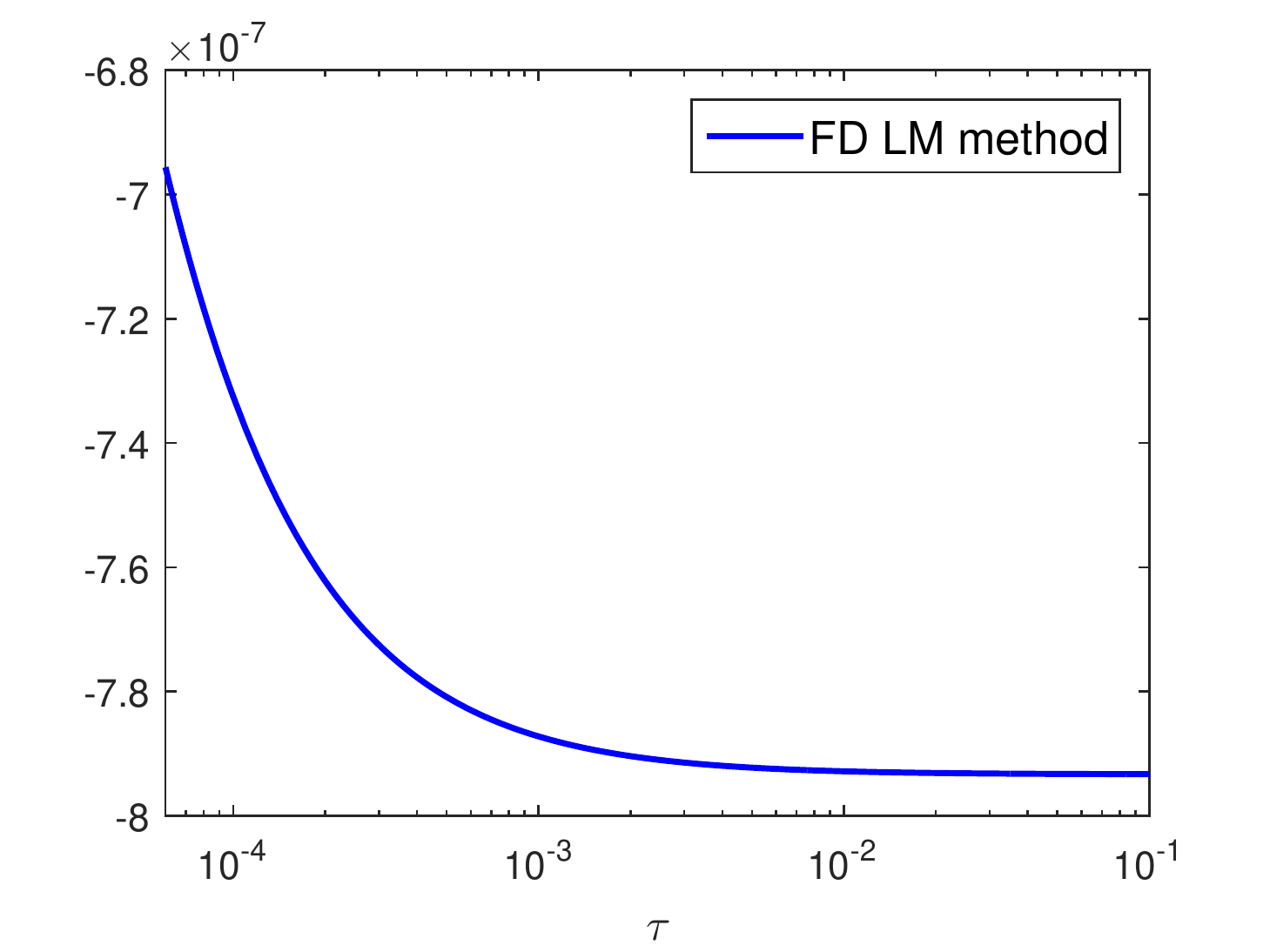}
}
\caption{Evolution of the smallest entry of $\E(t)$ for the
semidiscrete and fully discrete LM methods for Example (e), with  $h=0.141$. }
\label{fig:semi_e}
\end{figure}

\begin{table}[h]
\begin{center}{
\caption{Positivity thresholds for Example (e)
for the semidiscrete and fully discrete SG and FVE methods.}
\label{tab:semi-fully_e}
\begin{tabular}{|c|c|c|cc|cc|cc|}
\hline
   &  & & \multicolumn{2}{c|}{single-term ($\alpha$)}                          & \multicolumn{2}{c|}{multi-term ($\alpha_1,\alpha_2$)}                           & \multicolumn{2}{c|}{distributed $(\mu(\alpha))$}                    \\ \cline{4-9}
&$h_0$ &$h$  & \multicolumn{2}{c|}{$0.5$}  & \multicolumn{2}{c|}{$(0.5,0.2)$}  & \multicolumn{2}{c|}{$e^\al$}                  \\ \cline{4-9}
 &  &  & $\hat t_0\backslash\hat \tau_0$      & $\tilde t_0\backslash\tilde \tau_0$         &$\hat t_0\backslash\hat \tau_0$    & $\tilde t_0\backslash\tilde \tau_0$
   & $\hat t_0\backslash\hat \tau_0$       & $\tilde t_0\backslash\tilde \tau_0$             \\
   \hline
  &   0.100 & 0.141&  1.96e-4    &1.47e-4 &2.11e-4   &1.56e-4  &1.64e-2   &1.49e-2           \\ 
SD  & 0.050 &0.071 &  3.75e-5    &2.75e-5 &3.87e-5   &2.83e-5  &1.04e-2   &9.12e-3           \\ 
  &   0.025 & 0.035 & 1.58e-5    &1.20e-5 &1.66e-5    &1.24e-5  &5.43e-3  &4.72e-3     \\ \hline
  &   0.100 & 0.141 & 2.86e-5 &1.98e-5 &3.11e-5    &2.14e-5  &2.02e-3  &1.63e-3           \\
FD &  0.050 & 0.071 & 1.82e-6 &1.26e-6 &1.89e-6   &1.31e-6  &4.17e-4  &3.40e-4           \\ 
  &   0.025  & 0.035 & 1.36e-7  &8.21e-8 &1.41e-7    &8.31e-8  &9.80e-5   &7.42e-5     \\ \hline
\end{tabular}}
\end{center}
\end{table}

\appendix

\section{Convergence of the time stepping scheme}

In the proof of Theorem \ref{thm:semi-lm} and also
in Appendix B we need the convergence of the
fully discrete solution matrix  to the semidiscrete one
as the time step tends to zero, as  shown in
the following lemma.
\begin{lemma}\label{4conv}
Let $\E(t)$ and $\E_{n,\tau}$ be the semidiscrete and fully
discrete solution matrices, defined by \eqref{eqn:fem-matrix} and \eqref{eqn:BE_SG}, respectively. Then,
with the norm defined by the inner product $\M V\cdot W$, we have
\[
 \| \E(t)-\E_{n,t/n}   \| \le Cn^{-1},\quad \text{for}~~n\ge1.
\]
\end{lemma}
\begin{proof}
By  eigenvector expansion, cf. \eqref{3.E} and \eqref{4.E},
		and Parseval's formula we have
\[
 \| \E(t)-\E_{n,t/n} \| = \max_j |u_{\la_j}-r_{n,t/n}(\la_j)|,
\]
where $r_{n,\tau}(\la),\ n\ge1,$ with $\tau=t/n$, are the rational functions for which
$\E_{n,\tau}=r_{n,\tau}(\H)$, and
the $\la_j$ are the eigenvalues of $\H$.
So it remains to show the inequality
\[
  |u_{\la}(t)-r_{n,t/n}(\la)|\le Cn^{-1},
 \for~~\la>0,~~n\ge1.
\]
We recall from \eqref{22} that
\begin{equation}\label{eqn:ula}
    u_{\la}(t) = \frac{1}{2\pi\mathrm{i}}\int_{\Gamma_\sigma}J_\la(z) e^{zt}\,dz
    \with J_\la(z)=z^{-1}\,P(z)(P(z)+\la)^{-1}.
\end{equation}
To determine a corresponding representation of $r_{n,\tau}(\la)$ we
consider
the fully discrete solution $u^n=r_{n,\tau}(\la),\ n\ge0,$
of the scalar problem \eqref{scalder} for $u_\la(t)$,
which is given by
\begin{equation}\label{eqn:disc}
  \sum_{j=0}^n \omega_{n-j}u^j  + \la u^n =
  \sum_{j=0}^n \omega_{n-j} \for n\ge1,\with\ u^0=1.
\end{equation}
Multiplying both sides by $\xi^n$ and summing over
$n$ from $1$ to $\infty$, we obtain
\begin{equation*}
\sum_{n=1}^\infty \xi^n\sum_{j=0}^n \omega_{n-j}u^j
 + \la \sum_{n=1}^\infty u^n\xi^{n}
= \sum_{n=1}^\infty\xi^n\sum_{j=0}^n\omega_{n-j}.
\end{equation*}
Upon adding $\omega_0$ to both sides, the summation in the double sums may
start at $n=0$, so that they are both discrete convolutions, and hence,
with $\widetilde u(\xi)$ the generating function of the $u^j$ (i.e., $\widetilde{u}(\xi)=\sum_{j=0}^\infty u^j\xi^j$),
$\delta_\tau(\xi)=(1-\xi)/\tau$ that of the backward Euler method,  and $(1-\xi)^{-1}$ that
 of the sequence $(1,1,1,\dots)$, we obtain
\begin{equation*}
P(\delta_\tau(\xi))\widetilde u(\xi)
+\la (\widetilde u(\xi)-1)=P(\delta_\tau(\xi))(1-\xi)^{-1},
\end{equation*}
or
\begin{equation*}
(P(\delta_\tau(\xi))+\la)(\widetilde u
(\xi)-1)=P(\delta_\tau(\xi))((1-\xi)^{-1}-1).
\end{equation*}
Consequently, we have
\begin{equation*}
    \widetilde u(\xi) -1 = (P(\delta_\tau(\xi))+\lambda)^{-1}
	 P(\delta_\tau(\xi))
	 \xi(1-\xi)^{-1}
	  = \tau^{-1}\xi J_\lambda(\delta_\tau(\xi)),
\end{equation*}
Hence, the $n$th term in $\widetilde u(\xi)$
is given by
\begin{equation*}
  u^n=\frac{1}{2\pi\mathrm{i}}\int_{|\xi|=\epsilon}\tau^{-1}
  J_\lambda(\delta_\tau(\xi))\xi^{-n}\,d\xi, \for n\ge1,
\end{equation*}
and changing variables with $z=(1-\xi)/\tau$ yields, with small $\epsilon>0$,
\begin{equation}\label{4.cf}
r_{n,\tau}(\lambda)=u^n=\frac{1}{2\pi\mathrm{i}}\int_{|1-z\tau|=\epsilon}
J_\lambda(z)(1-z\tau)^{-n}\,d\xi.
\end{equation}
Since $J_\la(\xi)$ is analytic in the complement of the
negative real axis we may deform the contour of integration into
$\Gamma_\sigma$ and, with the change of variables
 $z=\zeta/\tau,\ \tau=t/n$,
we obtain the error representation
\begin{equation*}
    u_{\la}(t)-r_{n,\tau}(\la)=
    \frac{1}{2\pi\mathrm{i}}\int_{\Gamma_\sigma} J_\la(z)\left(e^{zt}-(1-z\tau)^{-n}\right)\,dz
   = \frac{1}{2\tau\pi\mathrm{i}}\int_{\Gamma_\sigma} J_\la(\zeta/\tau)
   \left(e^{n\zeta}-(1-\zeta)^{-n}\right)\,d\zeta.
\end{equation*}
We now  deform $\Gamma_\sigma$ into
	$\Gamma=\Gamma_{0}\cup\Gamma_{1}$, with
$\Gamma_0=\{ e^{\mathrm{i}\fy}/(2n): \fy\in[-\theta,\theta]\}$
a circular arc
and
$\Gamma_1=\{ \rho e^{\pm\mathrm{i}\theta}:\rho\ge 1/(2n)\}$,
a pair of rays in the left half-plane,
where $\theta\in(\pi/2,\pi)$. On  $\Gamma_0$,
we have $|\zeta|=1/(2n)$ and $|(1-\zeta)^{-1}|\leq e^{c|\zeta|}$
and hence
\begin{equation*}
  |e^{n\zeta}-(1-\zeta)^{-n}|  =
	\bigg|\left(e^{\zeta}-(1-\zeta)^{-1}\right)
	\sum_{j=0}^{n-1}(1-\zeta)^{-j}e^{(n-1-j)\zeta} \bigg|
	\le C|\zeta|^2\,n\,e^{cn|\zeta|}\le Cn^{-1}.
\end{equation*}
Next, we write   $\Gamma_1=\Gamma_R\cup\Gamma^R$ where
$\Gamma_R=\{ \rho e^{\pm\mathrm{i}\theta}:\rho\in[1/(2n),R]  \}$
and $\Gamma^R=\{ \rho e^{\pm\mathrm{i}\theta}:\rho\ge R  \}$.
With $c>0$ and  $R$ large enough such that
$|(1-\zeta)^{-1}| \le e^{-c}$
on $\Gamma^R$, using also
$|(1-\zeta)^{-1}| \le |\zeta|^{-1}$ on  $\Gamma_1$,
we have
\begin{equation*}\label{eqn:GR1}
  |e^{n\zeta}-(1-\zeta)^{-n}|  \le |e^{n\zeta}|+|1-\zeta|^{-n}
  \le C|\zeta|^{-1} e^{-cn},
\for \zeta\in \Gamma^R.
\end{equation*}
Further, with $R>0$ given, we have
$|(1-\zeta)^{-1}|\le e^{-c|\zeta|}$
for $\zeta\in\Gamma_R$
and $c>0$ small enough,
and  hence
\begin{equation*}
  |e^{n\zeta}-(1-\zeta)^{-n}|  =
	\big|\left(e^{\zeta}-(1-\zeta)^{-1}\right)
	\sum_{j=0}^{n-1}(1-\zeta)^{-j}e^{(n-1-j)\zeta} \big|
	\le C|\zeta|^2\,n\,e^{-cn|\zeta|},
	\for \zeta\in\Gamma_R.
\end{equation*}
We also note that $|J_\la(z)|\le C/|z|$
for $z\in\Sigma_\theta=\{z:\, z\neq0,\, |\arg z|\le\theta\}$ with $\theta<\pi$.
In fact, if $z\in \Sigma_\theta$,
	with $\Im z\ge0$, say, then the same holds
for $z^\al$ if $\al\in[0,1]$ and
$w=P(z)=\int_0^1z^\al\,d\nu(\al)$, and it
hence suffices to show $|w/(w+\lambda)|\le C$ for $w\in \Sigma_{\theta}$.
For $\mathrm{Re}\, w\ge0$ we have $|w/(w+\lambda)|\le1$,
and for $\mathrm{Re}\, w<0$,
$|w/(w+\lambda)|\le |w|/|\Im w|\le1/|\cos\theta|$.
Together these inequalities yield
\begin{equation*}
\begin{split}
   | u_{\la}(t)-r_{n,t/n}(\la)|&\le C\Big(n^{-1}
	\int_{\Gamma_0}|\zeta|^{-1} \,|d\zeta|
   +  n\int_{\Gamma_R}|\zeta|e^{-cn|\zeta|} \,|d\zeta|
   +e^{-cn}\int_{\Gamma^R} |\zeta|^{-2} \,|d\zeta|\Big)
   \\
  & \le C\Big(n^{-1}\int_{-\theta}^\theta \,d\psi
   +  \int_{1/2n}^R nr e^{-cnr} \, dr
   +e^{-cn}\int_R^\infty r^{-2}\,dr\Big)\le Cn^{-1}.
\end{split}
\end{equation*}
This completes the proof of the lemma.
\end{proof}

We observe from \eqref{4.cf} that by
Cauchy's integral formula,
$ r_{n,\tau}(\la) = (-1)^{n-1} \tau^{-n}J_\la^{(n-1)}(1/\tau) $.

\section{Maximum-norm contractivity of the lumped mass method}
The maximum principle for
\eqref{eqn:fde} ensures the maximum-norm contractivity, i.e.,
\begin{equation}
	\label{A1}
  \| E(t)v \|_{L^\infty(\Omega)}\le
  \| v \|_{L^\infty(\Omega)}\quad \forall t \ge 0.
\end{equation}
It is natural to ask if this property remains valid for discrete methods,
and in \cite{ThomeeWahlbin:2008,SchatzThomeeWahlbin:2010}
this problem was discussed for the SG and LM
methods for the heat equation.
It turned out that, in contrast to the continuous case,
the discrete analogue of \eqref{A1} is not equivalent
to the preservation of nonnegativity.
In this appendix, we discuss  maximum-norm contractivity for the LM method.
In this case the desired property holds when the stiffness matrix $\S$ is
diagonally dominant.

\medskip

We first show a contraction property for the fully discrete method,
and write $\|V\|_\infty=\max_i|V_i|.$
\begin{theorem}\label{thm:LM-contraction-discrete}
The solution matrix $\E_{n,\tau}$ of the fully discrete LM method is contractive in the maximum-norm
if  the matrix $\calS$ is diagonally dominant. If $\E_{1,\tau}$ is contractive in the maximum-norm,
then $\S$ is diagonally dominant.
\end{theorem}
\begin{proof}
Assume that $\calS$ is diagonally dominant.
Since $\calM$ is positive and diagonal,
 $\calH=\calM^{-1}\calS$
is row diagonally dominant.
Let
 $V  \in \mathbb{R}^N$,
set $U^1=\E_{1,\tau}V=
(\calI+\omega_0^{-1}\calH)^{-1}V$, and let
$|U^1_j|=\|U^1\|_\infty$. Then
\begin{equation*}
  (1+\omega_0^{-1}h_{jj})|U^1_j| =
  |V_j-\omega_0^{-1}\sum_{l\neq j} h_{jl}U^1_l|
  \leq \|V\|_\infty + \omega_0^{-1}{h}_{jj}\|U^1\|_\infty,
\end{equation*}
and hence
$\|U^1\|_\infty\leq \|V\|_\infty$, i.e.,
$\|(\calI+\omega_0^{-1}\calH)^{-1}\|_\infty \leq 1$.

We now prove that  $ \E_{n,\tau}$ is a contraction, by showing
by induction that,
for all $V  \in \mathbb{R}^N$, $
\| \E_{n,\tau} V \|_{\infty} = \|V^n\|_\infty \le \| V \|_{\infty}$.
To this end,
assume that this holds
for all $j\le  n-1$.
Using  Lemma \ref{lem:quad1prop}, we then find
\begin{equation*}
  \begin{aligned}
 \|  \sum_{j=0}^{n-1}\omega_{j} \ V -
 \sum_{j=1}^{n-1} \omega_{n-j} U^{j}\|_\infty
  \leq
    \sum_{j=0}^{n-1} \omega_j
\    \|V\|_\infty + \sum_{j=1}^{n-1}(-\omega_j)
\|V\|_\infty= \omega_0\|V\|_\infty,
  \end{aligned}
\end{equation*}
and hence,
cf. \eqref{eqn:BE_SG}, we get
\begin{equation*}
\begin{split}
   \| U^n   \|_\infty =  
 \Big | \hspace{-0.5mm}\Big | (\omega_0\calI+  \calH)^{-1}
 \Big (\sum_{j=0}^{n-1}\omega_{j}\ V -\sum_{j=1}^{n-1}
 \omega_{n-j} U^{j}\Big) \Big | \hspace{-0.5mm}\Big | _\infty
    \le \| (\calI+\omega_0^{-1} \calH)^{-1}  \|_\infty
    \| V  \|_{\infty}\le \|V\|_\infty,
\end{split}
\end{equation*}
which shows our claim.

Now,
suppose $\E_{1,\tau}$ is a contraction. In view of \eqref{eqn:firststep},
we have
\begin{equation*}
\|\E_{1,\tau}\|_\infty =\max_i\sum_j|(\E_{1,\tau})_{ij}|
=\max_i\Big(1-\wt\beta_0(\tau)\big(h_{ii}-\sum_{j\ne i}|h_{ij}|\big)
+o(\wt\beta_0(\tau))\Big),\as \tau\to0.
\end{equation*}
If $\|\E_{1,\tau}\|_\infty\le1$,  by taking $\tau$ small we find
$\sum_{j\neq i}| h_{ij}|\leq h_{ii}$ for
$i=1,\ldots,N$. Hence $ \calH$ is row diagonally
dominant.
Since $ \calM$ is a positive  diagonal matrix,
$\calS=\calM \calH$ is also row diagonally
dominant and hence, since it is symmetric, diagonally dominant.
\end{proof}

The next result gives the contractivity of the semidiscrete LM method.
\begin{theorem}\label{thm:LM-contraction}
The semidiscrete LM solution matrix $\E(t)$ is contractive in the maximum
norm if  $\calS$ is diagonally dominant. If $\E(t)$ is contractive
for small $t$ then $\S$ is diagonally dominant.
\end{theorem}
\begin{proof}
By Theorem \ref{thm:LM-contraction-discrete}, $\E_{n,\tau}$
is a contraction in the maximum-norm, and, by Lemma \ref{4conv},
converges
 to the semidiscrete solution matrix
in the sense that $\E(t)=\lim_{n\to\infty}
\E_{n,t/n}$.  Hence also
$\E(t)$ is a contraction.
Conversely, if $\E(t)$ is  a contraction  in
$\|\cdot\|_{\infty}$ for small $t$, then by
\eqref{3b0} we have
\begin{equation*}
 \|\E(t)\|_\infty
=\max_i\Big(1-\beta_0(t)\big(h_{ii}-\sum_{j\ne i}|h_{ij}|\big)
+o(\beta_0(t))\Big)\as t\to0.
\end{equation*}
The rest of the proof is then identical to that of
Theorem \ref{thm:LM-contraction-discrete}.
\end{proof}

\section*{Acknowledgements}
The work of the first author (B. Jin) is partly supported by
UK Engineering and Physical
Sciences Research Council grant EP/M025160/1.

 \medskip

 \noindent
Department of Computer Science \\
University College London\\
Gower Street, London WC1E 6BT, UK  \\
  e-mail: (bangti.jin@gmail.com,b.jin@ucl.ac.uk)\\

 \noindent 
Department of Mathematics\\
Texas A\&M University\\
College Station, TX 77843-3368, USA \\
e-mail: lazarov@math.tamu.edu\\

\noindent Mathematical Sciences\\
Chalmers University of Technology and the
University of Gothenburg\\
 SE-412 96 G\"oteborg, Sweden \\
e-mail: thomee@math.chalmers.se \\

\noindent 
Department of Applied Physics and Applied Mathematics\\
Columbia University\\
500 W. 120th Street, New York, NY 10027, USA \\
e-mail: zhizhou0125@gmail.com  \\

\begin{thebibliography}{10}


\bibitem{Chatz:2014}
P.~Chatzipantelidis, Z.~Horv{\'a}th, and V.~Thom{\'e}e.
\newblock On positivity preservation in some finite element methods for the
  heat equation.
\newblock {\em Comput. Methods Appl. Math.}, 15(4): 417--437, 2015.

\bibitem{ChatzLazarovThomee:2013}
P.~Chatzipantelidis, R.~Lazarov, and V.~Thom{\'e}e.
\newblock Some error estimates for the finite volume element method for a
  parabolic problem.
\newblock {\em Comput. Methods Appl. Math.}, 13(3):251--279, 2013.

\bibitem{ChechkinGorenfloSokolov:2002}
A.~V. Chechkin, R.~Gorenflo, and I.~M. Sokolov.
\newblock Retarding subdiffusion and accelerating superdiffusion governed by
  distributed-order fractional diffusion equations.
\newblock {\em Phys. Rev. E}, 66:046129, 2002.

\bibitem{ChouLi:2000}
S.-H. Chou and Q.~Li.
\newblock Error estimates in {$L^2,\ H^1$} and {$L^\infty$} in covolume methods
  for elliptic and parabolic problems: a unified approach.
\newblock {\em Math. Comp.}, 69(229):103--120, 2000.

\bibitem{Druaguanescu:2005}
A.~Dr{\u{a}}g{\u{a}}nescu, T.~Dupont, and L.~Scott.
\newblock Failure of the discrete maximum principle for an elliptic finite
  element problem.
\newblock {\em Math. Comp.}, 74(249):1--23, 2005.

\bibitem{Feller:2008}
W.~Feller.
\newblock {\em An {I}ntroduction to {P}robability {T}heory and its
  {A}pplications}, Vol.~2.
\newblock John Wiley \& Sons, 1971.

\bibitem{Fujii:1973}
H.~Fujii.
\newblock Some remarks on finite element analysis of time-dependent field
  problems.
\newblock In. Theory and Practice in Finite Element Structural Analysis (Y.
  Yamada, R. H. Gallagher \& N. K. Kyokai eds). Tokyo, Japan: University of
  Tokyo Press, pp. 91--106., 1973.


\bibitem{JinLazarovLiuZhou:2015}
B.~Jin, R.~Lazarov, Y.~Liu, and Z.~Zhou.
\newblock The {G}alerkin finite element method for a multi-term time-fractional
  diffusion equation.
\newblock {\em J. Comput. Phys.}, 281:825--843, 2015.


\bibitem{JinLazarovSheenZhou:2015}
B.~Jin, R.~Lazarov, D.~Sheen, and Z.~Zhou.
\newblock Error estimates for approximations of distributed order time
  fractional diffusion with nonsmooth data.
\newblock preprint, arXiv:1504.01529, 2015.

\bibitem{JinLazarovZhou:2013}
B.~Jin, R.~Lazarov, and Z.~Zhou.
\newblock Error estimates for a semidiscrete finite element method for
  fractional order parabolic equations.
\newblock {\em SIAM J. Numer. Anal.}, 51(1):445--466, 2013.

\bibitem{JinLazarovZhou:2014}
B.~Jin, R.~Lazarov, and Z.~Zhou.
\newblock On two schemes for fractional diffusion and diffusion wave equations.
\newblock preprint, arXiv:1404.3800, 2014.

\bibitem{KilbasSrivastavaTrujillo:2006}
A.~Kilbas, H.~Srivastava, and J.~Trujillo.
\newblock {\em Theory and {A}pplications of {F}ractional {D}ifferential
  {E}quations}.
\newblock Elsevier, Amsterdam, 2006.

\bibitem{Kochubei:2008}
A.~N. Kochubei.
\newblock Distributed order calculus and equations of ultraslow diffusion.
\newblock {\em J. Math. Anal. Appl.}, 340(1):252--281, 2008.

\bibitem{LiLiuYamamoto:2015}
Z.~Li, Y.~Liu, and M.~Yamamoto.
\newblock Initial-boundary value problems for multi-term time-fractional
  diffusion equations with positive constant coefficients.
\newblock {\em Appl. Math. Comput.}, 257:381--397, 2015.

\bibitem{LiLuchkoYamamoto:2014}
Z.~Li, Y.~Luchko, and M.~Yamamoto.
\newblock Asymptotic estimates of solutions to initial-boundary-value problems
  for distributed order time-fractional diffusion equations.
\newblock {\em Fract. Calc. Appl. Anal.}, 17(4):1114--1136, 2014.

\bibitem{LinXu:2007}
Y.~Lin and C.~Xu.
\newblock Finite difference/spectral approximations for the time-fractional
  diffusion equation.
\newblock {\em J. Comput. Phys.}, 225(2):1533--1552, 2007.


\bibitem{Lubich:1988}
C.~Lubich.
\newblock Convolution quadrature and discretized operational calculus. {I}.
\newblock {\em Numer. Math.}, 52(2):129--145, 1988.

\bibitem{Luchko:2009fcaa}
Y.~Luchko.
\newblock Boundary value problems for the generalized time fractional diffusion
  equation of distributed order.
\newblock {\em Frac. Cal. Appl. Anal.}, 12(4):409--422, 2009.

\bibitem{Luchko2009}
Y.~Luchko.
\newblock Maximum principle for the generalized time-fractional diffusion
  equation.
\newblock {\em J. Math. Anal. Appl.}, 351(1):218--223, 2009.

\bibitem{Luchko:2011jmaa}
Y.~Luchko.
\newblock Initial-boundary problems for the generalized multi-term
  time-fractional diffusion equation.
\newblock {\em J. Math. Anal. Appl.}, 374(2):538--548, 2011.



\bibitem{MustaphaAbdallahFurati:2014}
K.~Mustapha, B.~Abdallah, and K.~M. Furati.
\newblock A discontinuous {P}etrov-{G}alerkin method for time-fractional
  diffusion equations.
\newblock {\em SIAM J. Numer. Anal.}, 52(5):2512--2529, 2014.


\bibitem{Persson:2004}
P.~Persson and G.~Strang.
\newblock A simple mesh generator in {MATLAB},
\newblock {\em SIAM Rev.}, 46(2):329--345, 2004.


\bibitem{Pollard:1948}
H.~Pollard.
\newblock The completely monotonic character of the {M}ittag-{L}effler function
  {$E_a(-x)$}.
\newblock {\em Bull. Amer. Math. Soc.}, 54:1115--1116, 1948.

\bibitem{SakamotoYamamoto:2011}
K.~Sakamoto and M.~Yamamoto.
\newblock Initial value/boundary value problems for fractional diffusion-wave
  equations and applications to some inverse problems.
\newblock {\em J. Math. Anal. Appl.}, 382(1):426--447, 2011.

\bibitem{SchatzThomeeWahlbin:2010}
A.~H. Schatz, V.~Thom{\'e}e, and L.~B. Wahlbin.
\newblock On positivity and maximum-norm contractivity in time stepping methods
  for parabolic equations.
\newblock {\em Comput. Methods Appl. Math.}, 10(4):421--443, 2010.



\bibitem{shewchuk1996triangle}
J.~R. Shewchuk.
\newblock Triangle: {E}ngineering a 2D quality mesh generator and delaunay
  triangulator.
\newblock In M.~C. Lin and D.~Manocha, editors, {\em Applied {C}omputational
  {G}eometry: {T}owards {G}eometric {E}ngineering}, pages 203--222. Springer,
  1996.


\bibitem{Thomee:2006}
V.~Thom{\'e}e.
\newblock {\em Galerkin {F}inite {E}lement {M}ethods for {P}arabolic
  {P}roblems},
\newblock Springer-Verlag, Berlin, 2006.

\bibitem{Thomee:2015}
V.~Thom\'{e}e.
\newblock On positivity preservation in some finite element methods for the
  heat equation.
\newblock In {I. Dimov et al.}, editor, {\em NMA 2014, LNCS 8962}, pp. 13--24.
  2015.

\bibitem{ThomeeWahlbin:2008}
V.~Thom{\'e}e and L.~B. Wahlbin.
\newblock On the existence of maximum principles in parabolic finite element
  equations.
\newblock {\em Math. Comp.}, 77(261):11--19, 2008.

\bibitem{WeidemanTrefethen:2007}
J.~A.~C. Weideman and L.~N. Trefethen.
\newblock Parabolic and hyperbolic contours for computing the {B}romwich
  integral.
\newblock {\em Math. Comp.}, 76(259):1341--1356, 2007.

\end{thebibliography}
\end{document}